\documentclass[letterpaper,11pt]{article}

\usepackage{amsmath,amsthm}
\usepackage{amsfonts}
\usepackage{latexsym}
\usepackage{amssymb,bm}
\usepackage{mathrsfs}
\usepackage{mathtools}
\usepackage[page]{appendix}
\usepackage{float}
\usepackage{changes}

\usepackage{graphicx}
\usepackage{epstopdf}

\usepackage{tikz}

\usepackage{newfloat}
\usepackage{caption}
\usepackage{placeins}
\DeclareFloatingEnvironment[fileext=frm,placement={ht},name=Frame]{algfloat}

 \usepackage{grffile}

\usepackage{changes}

\usepackage{color}

\newtheorem{theorem}{Theorem}[section]
\newtheorem{lemma}[theorem]{Lemma}
\newtheorem{corollary}[theorem]{Corollary}
\newtheorem{proposition}[theorem]{Proposition}

\newtheorem{definition}[theorem]{Definition}

\newtheorem{algorithm}{Algorithm}

\newcommand{\inner}[2]{\langle #1,#2\rangle}

\newcommand{\norm}[1]{\|{#1}\|}

\newcommand{\R}{\mathbb{R}}

\newcommand{\comenta}[1]{}

\newcommand{\HH}{\mathcal{H}}

\newcommand{\mgap}{\vspace{.1in}}

\begin{document}

\title{Variants of the A-HPE and large-step A-HPE algorithms for strongly convex problems with applications 
to accelerated high-order tensor methods}
\author{
M. Marques Alves
\thanks{
Departamento de Matem\'atica,
Universidade Federal de Santa Catarina,
Florian\'opolis, Brazil, 88040-900 
({\tt maicon.alves@ufsc.br}).
The work of this author was partially 
supported by CNPq grants no. 304692/2017-4.
}
}

\maketitle

\begin{abstract}
For solving strongly convex optimization problems, we propose and study the global convergence of variants of the 
accelerated hybrid proximal extragradient (A-HPE) and large-step A-HPE algorithms of Monteiro and Svaiter~\cite{mon.sva-acc.siam13}.
We prove \emph{linear} and the \emph{superlinear} $\mathcal{O}\left(k^{\,-k\left(\frac{p-1}{p+1}\right)}\right)$ global rates for the proposed
variants of
the A-HPE and large-step A-HPE methods, respectively. 
The parameter $p\geq 2$ appears in the (high-order) large-step condition of the new large-step A-HPE algorithm.
 We apply our results to high-order tensor methods, obtaining a new inexact (relative-error)
 tensor method for (smooth) strongly convex optimization with iteration-complexity $\mathcal{O}\left(k^{\,-k\left(\frac{p-1}{p+1}\right)}\right)$.
 In particular, for $p=2$, we obtain an inexact proximal-Newton algorithm with fast global
 $\mathcal{O}\left(k^{\,-k/3}\right)$ convergence rate.
 \\
  \\
  2000 Mathematics Subject Classification: 90C60, 90C25, 47H05, 65K10.
  \\
  \\
  Key words: Convex optimization, strongly convex, accelerated methods, proximal-point algorithm, large-step, 
  high-order tensor methods, superlinear convergence, proximal-Newton method.
 \end{abstract}

\pagestyle{plain}


\section{Introduction}
 \label{sec:int}
The \emph{proximal-point method}~\cite{martinet70,roc-mon.sjco76} is one of the most popular algorithms for solving nonsmooth convex optimization problems. 
For the general problem of minimizing a convex function $h(\cdot)$, its \emph{exact} version can be described by the iteration
\begin{align}
 \label{eq:pp_intr}
 x^{k+1}=\mbox{arg}\min_{x}\,\left\{h(x)+\dfrac{1}{2\lambda}\norm{x-x^k}^2\right\},\qquad k\geq 0,
\end{align}
where $\lambda=\lambda_{k+1}>0$ and $x^k$ is the current iterate.
Motivated by the fact that in many cases the computation of $x^{k+1}$ is numerically expensive, several authors have proposed \emph{inexact} versions of \eqref{eq:pp_intr}. 
Among them, inexact proximal-point methods based on \emph{relative-error} criterion for the subproblems are currently quite popular. 
For the more abstract setting of solving inclusions for maximal monotone 
operators, this approach was initially developed by Solodov and Svaiter (see, e.g., \cite{sol.sva-hpe.svva99,sol.sva-hyb.jca99, sol.sva-breg.mor00,sol.sva-uni.nfao01}), subsequently studied, from the viewpoint of computational complexity, by 
Monteiro and Svaiter 
(see, e.g., \cite{mon.sva-tse.siam10,mon.sva-hpe.siam10, mon.sva-newton.siam12,mon.sva-acc.siam13}) and has gained a lot of attention by different authors and research groups 
(see, e.g., \cite{att.alv.sva-dyn.jca16,bach-preprint21,eck.yao-rel.mp18,zhang-preprint20,jordan-controlpreprint20}) with many applications 
in optimization algorithms and related topics such as variational inequalities, saddle-point problems, etc.

The starting point of this contribution is \cite{mon.sva-acc.siam13}, where the 
relative-error inexact hybrid proximal extragradient (HPE) method~\cite{mon.sva-hpe.siam10,sol.sva-hpe.svva99} 
was accelerated for convex optimization, by using Nesterov's acceleration~\cite{nes-book}.
The resulting accelerated HPE-type algorithms, called A-HPE and large-step A-HPE, were applied to first- and second-order optimization, with iteration-complexities $\mathcal{O}\left(1/k^2\right)$ and 
$\mathcal{O}\left(1/k^{7/2}\right)$, respectively. The A-HPE and/or the large-step A-HPE algorithms were recently studied also in 
\cite{arjevani-oracle.mp19,bach-preprint21,bubeck-near.pmlr19,gasnikov-optimal.pmlr19,zhang-preprint20,jordan-controlpreprint20}, 
with applications in high-order optimization, machine learning and tensor methods.

In this paper, we consider the (unconstrained) convex optimization problem
\begin{align}
 \label{eq:co}
  \min_{x}\,\left\{h(x):=f(x)+g(x)\right\},
\end{align}
where $f$ is convex and $g$ is \emph{strongly convex}. 
For solving \eqref{eq:co}, we propose and study the convergence rates of variants of the A-HPE and large-step A-HPE algorithms. The new algorithms are designed especially for strongly convex problems, and the resulting global convergence rates are \emph{linear} and 
$\mathcal{O}\left(k^{\,-k\left(\frac{p-1}{p+1}\right)}\right)$ for the variants of the A-HPE and large-step A-HPE, respectively.
(the parameter $p\geq 2$ appears in the high-order large-step condition (see \cite{zhang-preprint20,jordan-controlpreprint20}.)
We also apply our study to tensor algorithms for high-order convex optimization, a topic which has been the object of investigation of several authors 
(see, e.g., \cite{bubeck-near.pmlr19,doikov-local.preprint19,gra.nes-ten.oms20,zhang-preprint20,jordan-controlpreprint20,nes-preprint20b,nes-preprint20a} and references therein).
The proposed inexact (relative-error) $p$-th order tensor algorithm has global \emph{superlinear} $\mathcal{O}\left(k^{\,-k\left(\frac{p-1}{p+1}\right)}\right)$ convergence rate. We also mention that, for $p=2$ we obtain, as a by-product of our approach to high-order optimization, a fast $\mathcal{O}\left(k^{-k/3}\right)$ proximal-Newton method for strongly convex optimization.

The main contributions of this paper can be summarized as follows:

\begin{itemize}
\item[(i)] A variant of the A-HPE algorithm for strongly convex objectives (Algorithm \ref{alg:main}) and its iteration-complexity analysis as in Theorems \ref{th:main_comp} and \ref{th:main_comp03}.
\item[(ii)] A large-step A-HPE-type algorithm for strongly convex problems (Algorithm \ref{alg:ls}) with a high-order large-step condition and its iteration-complexity (see Theorem \ref{th:main_comp02}).
\item[(iii)] A new inexact high-order tensor algorithm (Algorithm \ref{alg:second}) for strongly convex problems and its global convergence analysis (see Theorem \ref{th:main_tensor}). Here and in item (ii) above we highlight the fast global convergence rate $\mathcal{O}\left(k^{\,-k\left(\frac{p-1}{p+1}\right)}\right)$.
\item[(iv)] An inexact relative-error forward-backward algorithm for strongly convex optimization 
(see Algorithm \ref{alg:first01} and Theorem \ref{th:first01}). 
\end{itemize}

Additionally to the contributions described in (i)--(iv) above, we refer the reader to the remarks/comments following Algorithms \ref{alg:main}, \ref{alg:ls}, \ref{alg:second} and \ref{alg:first01}. 

\mgap
\noindent
{\bf Some previous contributions.} 
The A-HPE and forward-backward methods for strongly convex problems were also recently studied in \cite{bach-preprint21}.
Based on the A-HPE framework, $p$th-order tensor methods with iteration-complexity $\mathcal{O}\left(1/k^{\frac{3p+1}{2}}\right)$ were studied in  
\cite{arjevani-oracle.mp19,bubeck-near.pmlr19,gasnikov-optimal.pmlr19,zhang-preprint20,jordan-controlpreprint20}.
When combined with restart techniques, improved rates for the uniformly- and/or strongly convex case were also obtained in 
\cite{arjevani-oracle.mp19,gasnikov-optimal.pmlr19} (see also \cite{kornowski-high.preprint20}). 
We also mention that local superlinear convergence rates for tensor methods were obtained in \cite{doikov-local.preprint19}.
Tensor and/or second-order schemes were also recently studied in 
\cite{doi.nes-min.jota21,dvurechensky-near.preprint19,gra.nes-acc.siam19,gra.nes-ten.oms20}.

\mgap

\noindent
{\bf General notation.} 
We denote by $\HH$ a finite-dimensional real vector space with inner product $\inner{\cdot}{\cdot}$ and induced norm $\norm{\cdot}=\sqrt{\inner{\cdot}{\cdot}}$.
The \emph{$\varepsilon$-subdifferential} and the \emph{subdifferential} of a convex function 
$g:\HH\to (-\infty,\infty]$ at $x\in \HH$ are defined as
$\partial_\varepsilon g(x) := \{u\in \HH\;|\; g(y)\geq g(x)+\inner{u}{y-x}-\varepsilon\quad \forall y\in \HH\}$ and 
$\partial g(x) := \partial_0 g(x)$, respectively. For additional details on standard notations and definitions of convex analysis we refer the reader to the reference~\cite{rock-ca.book}.
Recall that $g:\HH\to (-\infty,\infty]$ is $\mu$-strongly convex if
 $\mu>0$ and, for all $x,y\in \HH$,
\begin{align}
\label{eq:def.str}
g(\lambda x+(1-\lambda)y)\leq \lambda g(x)+(1-\lambda)g(y)-
\dfrac{1}{2} \mu\lambda(1-\lambda)
\norm{x-y}^2,\qquad \forall 
\lambda\in [0,1].
\end{align}

\section{A variant of the A-HPE algorithm for strongly convex problems}
 \label{sec:alg}

In this section, we consider the convex optimization problem \eqref{eq:co}, i.e.,
\begin{align*}
 %
  \min_{x\in \HH}\,\{h(x):=f(x)+g(x)\},
\end{align*}
where $f,g:\HH\to (-\infty, \infty]$ are proper, closed and convex functions, $\mbox{dom}\,h\neq\emptyset$, and 
$g$ is {\it $\mu$-strongly convex}, for some $\mu>0$. We will denote by $x^*$ the unique solution of \eqref{eq:co}. 

\mgap

Next we present the main algorithm of this section for solving \eqref{eq:co}, whose the complexity analysis will be presented in Theorems \ref{th:main_comp} and \ref{th:main_comp03}.

\mgap
\mgap

%
\noindent
\fbox{
\begin{minipage}[h]{6.6 in}
\begin{algorithm}
\label{alg:main}
{\bf A variant of the A-HPE algorithm for solving the (strongly convex) problem \eqref{eq:co}}
\end{algorithm}
\begin{itemize}
\item [0)] 
Choose $x^0,y^0\in \HH$, $\sigma\in [0,1]$, let $A_0=0$ and set $k=0$.
\item [1)] 
Compute $\lambda_{k+1}>0$ and $(y^{k+1},v^{k+1},\varepsilon_{k+1})\in \HH\times \HH\times \R_{++}$ such that 
\begin{align}
\begin{aligned}
 \label{eq:alg_err}
  &v^{k+1}\in \partial_{\varepsilon_{k+1}} f(y^{k+1})+ \partial g(y^{k+1}),\\[3mm]
  &\dfrac{\norm{\lambda_{k+1}v^{k+1}+y^{k+1}-\widetilde x^k}^2}{1+\lambda_{k+1}\,\mu}
  +2\lambda_{k+1}\varepsilon_{k+1}
  \leq \sigma^2\norm{y^{k+1}-\widetilde x^k}^2,
\end{aligned}
\end{align}
where
\begin{align}
\label{eq:alg_xtil}
& \widetilde x^k =  \left(\dfrac{a_{k+1}-\mu A_k\lambda_{k+1}}{A_k+a_{k+1}}\right)x^k
 + \left(\dfrac{A_k+\mu A_k\lambda_{k+1}}{A_k+a_{k+1}}\right)y^k,\\[3mm]
 \label{eq:alg_a}
& a_{k+1}=\dfrac{(1+2\mu A_k)\lambda_{k+1}+\sqrt{(1+2\mu A_k)^2\lambda_{k+1}^2
+4(1+\mu A_k)A_k\lambda_{k+1}}}{2}. 
\end{align}

\item[2)] 
 Let
  \begin{align}
  \label{eq:alg_A}
  &A_{k+1} = A_k + a_{k+1},\\[2mm]
  \label{eq:alg_xne}
  & x^{k+1} = \left(\dfrac{1+\mu A_k}{1+\mu A_{k+1}}\right)x^k + 
  \left(\dfrac{\mu a_{k+1}}{1+\mu A_{k+1}}\right) y^{k+1}
  - \left(\dfrac{a_{k+1}}{1+\mu A_{k+1}}\right)v^{k+1}.
  \end{align}
\item[3)] 
Set $k=k+1$ and go to step 1.
\end{itemize}
\noindent
\end{minipage}
} 
%

\mgap
\mgap

Next we make some remarks about Algorithm \ref{alg:main}:

\begin{itemize}
\item[(i)] By letting $\mu=0$ in Algorithm \ref{alg:main}, we obtain a special instance of the A-HPE algorithm of Monteiro and Svaiter (see \cite[Section 3]{mon.sva-acc.siam13}), whose global convergence rate is 
$\mathcal{O}\left(1/k^2\right)$ (see \cite[Theorem 3.8]{mon.sva-acc.siam13}).
On the other hand, thanks to the strong-convexity assumption on $g$, in Theorems \ref{th:main_comp} and \ref{th:main_comp03} we obtain \emph{linear convergence} for Algorithm \ref{alg:main}. 
We will also study a \emph{high-order} large-step version of Algorithm \ref{alg:main} (see Algorithm \ref{alg:ls} in Section \ref{sec:ls}), for
which \emph{superlinear} $\mathcal{O}\left(k^{\,-k\left(\frac{p-1}{p+1}\right)}\right)$ global convergence rates are proved, where $p\geq 2$. 
Applications of the latter result to high-order tensor methods for convex optimization will also be discussed in Section \ref{sec:ls}.
%
%
\item[(ii)] Since the cost of computing $\widetilde x^k, a_{k+1}, A_{k+1}$ and $x^{k+1}$ as in 
\eqref{eq:alg_xtil}--\eqref{eq:alg_xne} is negligible (from a computational viewpoint), it follows that
the computational burden of Algorithm \ref{alg:main} is represented by the computation of $\lambda_{k+1}>0$ and 
$(y^{k+1}, v^{k+1},\varepsilon_{k+1})$  as in \eqref{eq:alg_err}. 
In this regard, note that if $\mbox{prox}_{\lambda h}:=(\lambda \partial h+I)^{-1}$ of $h$ is computable, for $\lambda>0$, then $\lambda_{k+1}:=\lambda$ and 
$(y^{k+1}, v^{k+1},\varepsilon_{k+1}):=\left(\mbox{prox}_{\lambda h}(\widetilde x^k),\frac{\widetilde x^k-y^{k+1}}{\lambda_{k+1}},0\right)$ clearly satisfy the conditions in \eqref{eq:alg_err} with $\sigma=0$.
On the other hand, in the more general setting of $\sigma>0$, Algorithm \ref{alg:main} can be used both as a framework for the design and analysis of practical algorithms \cite{mon.sva-acc.siam13} and as a \emph{bilevel} method, in which the inequality
in \eqref{eq:alg_err} is used as a stopping criterion for some \emph{inner} algorithm applied to the regularized inclusion 
$0\in \lambda \partial h(x)+x-\widetilde x^k$. In this case, note that the \emph{error-criterion} in \eqref{eq:alg_err} 
is \emph{relative} and controlled by the parameter $\sigma\in (0,1]$.
%
%
\item[(iii)] We emphasize that the inequality in \eqref{eq:alg_err} is specially tailored for strongly convex problems, in the sense that it is more general than the usual inequality appearing in relative-error HPE-type 
methods (see, e.g., \cite{alves-reg.siam16,eck.yao-rel.mp18,mon.sva-hpe.siam10,mon.sva-newton.siam12,sol.sva-hpe.svva99}), which in the context of this paper would read as 
\[
 \norm{\lambda_{k+1} v^{k+1}+y^{k+1}-\widetilde x^k}^2+2\lambda_{k+1}\varepsilon_{k+1}
 \leq\sigma^2\norm{y^{k+1}-\widetilde x^k}^2.
\]
%
%
%
\item[(iv)] 
We also mention that Algorithm \ref{alg:main} is closely related to a variant of the A-HPE for strongly convex objectives presented and studied in \cite[Section 5]{bach-preprint21}. 
In this paper, by taking an approach similar to the one which was considered in~\cite{mon.sva-acc.siam13,nes-smo.mp05}, we obtain global convergence rates for 
Algorithm \ref{alg:main} in terms of \emph{function values}, \emph{sequences} and \emph{(sub-)gradients} 
(see Theorems \ref{th:main_comp} and \ref{th:main_comp03} below). In contrast to \cite{bach-preprint21}, in this paper we also consider a \emph{large-step} version of Algorithm \ref{alg:main}, namely Algorithm \ref{alg:ls}, for which the (global) superlinear $\mathcal{O}\left(k^{-k\,\left(\frac{p-1}{p+1}\right)}\right)$ convergence rate is proved (see Theorems \ref{th:main_comp02} and \ref{th:main_tensor} below).

\item[(v)] We note that condition \eqref{eq:alg_a} yields
\begin{align}
 \label{eq:cond.lamb}
 \dfrac{(1+\mu A_k)A_{k+1}\lambda_{k+1}}{a_{k+1}^2}
+\dfrac{\mu A_k\lambda_{k+1}}{a_{k+1}}=1.
\end{align}
Indeed, substitution of $A_{k+1}$ by $A_k+a_{k+1}$ (see \eqref{eq:alg_A}) and some simple algebra give that
\eqref{eq:cond.lamb} is equivalent to
\begin{align}
 \label{eq:cond.lamb2}
 a_{k+1}^2-(1+2\mu A_k)\lambda_{k+1}a_{k+1}-(1+\mu A_k)A_k\lambda_{k+1}=0.
\end{align} 
Note now that $a_{k+1}$ as in \eqref{eq:alg_a} is exactly the largest root of the quadratic equation in \eqref{eq:cond.lamb2}.

\item[(vi)] Using \eqref{eq:alg_A} and the fact that $A_0=0$ (see step 0) we obtain 
$A_1=A_0+a_1=a_1$. On the other hand, direct substitution of $A_0=0$ in \eqref{eq:alg_a} with $k=0$ yields
$a_1=\lambda_1$. As a consequence, we conclude that
\begin{align}
 \label{eq:A1a1}
 A_1=a_1=\lambda_1.
\end{align}

\end{itemize}

\mgap

In what follows in this section, we will analyze convergence rates of Algorithm \ref{alg:main}. 
To this end, we first define $\gamma_k(\cdot)$ and $\Gamma_k(\cdot)$ as, for all $x\in \HH$,
\begin{align}
\label{eq:def.gammak}
\gamma_k(x)=h(y^k)+\inner{v^k}{x-y^k}-\varepsilon_k+\dfrac{\mu}{2}\norm{x-y^k}^2\qquad (k\geq 1)
\end{align}
and
\begin{align}
 \label{eq:def.Gammak}
\Gamma_0(x)=0\;\;\mbox{and},\; \mbox{for}\;k\geq 1,\;\;\Gamma_k(x)=\sum_{j=1}^k\,\dfrac{a_j}{A_k}\gamma_j(x).
\end{align}
Note that 
\begin{align}
 \label{eq:der.gamma}
 \nabla \gamma_k(x)=v^k+\mu(x-y^k)\;\;\mbox{and}\;\; \nabla^2\gamma_k(x)=\mu I
\end{align}
and observe that $A_k$ ($k=0,1,\dots$) as in Algorithm \ref{alg:main} satisfies 
\begin{align}
 \label{eq:def.ak}
 A_0=0\;\;\mbox{and},\;\mbox{for}\;k\geq 1,\;\;
 A_{k}=\sum_{j=1}^k\,a_j.
\end{align}
From \eqref{eq:def.Gammak}--\eqref{eq:def.ak} we obtain, for $k\geq 1$,
\begin{align}
 \label{eq:der.Gamma}
  \nabla^2\Gamma_k(x)=\mu I,\qquad x\in \HH.
\end{align}
Note also that the following holds trivially from \eqref{eq:def.Gammak} and \eqref{eq:def.ak}: for all $k\geq 0$,
\begin{align}
 \label{eq:rec.ak}
&A_{k+1}\Gamma_{k+1} = A_k\Gamma_k+a_{k+1}\gamma_{k+1}.
\end{align}
Define also, for all $k\geq 0$,
\begin{align}
\label{eq:def.betak}
&\beta_k=\inf_{x\in \HH}\left\{A_k\Gamma_k(x)+\dfrac{1}{2}\norm{x-x^0}^2\right\}.
\end{align}
Note that $\beta_0=0$.

\mgap

The following three technical lemmas will be useful to prove the first result on the iteration-complexity of Algorithm \ref{alg:main}, namely Proposition \ref{pr:cambio} below.

\mgap

\begin{lemma}
 \label{lm:initial}
Let $\gamma_k(\cdot)$ and $\Gamma_k(\cdot)$ be as in \eqref{eq:def.gammak} and \eqref{eq:def.Gammak}, respectively.
The following holds:
\begin{itemize}
\item[\emph{(a)}] For all $k\geq 1$, we have $\gamma_k(x)\leq h(x),\quad \forall x\in \HH$.
\item[\emph{(b)}] For all $k\geq 0$, we have $x^k=\arg\min_{x\in \HH}\{A_k\Gamma_k(x)+\frac{1}{2}\norm{x-x^0}^2\}$.
\end{itemize}
\end{lemma}
\begin{proof}
(a) In view of the inclusion in \eqref{eq:alg_err} we have, for all $k\geq 1$, 
$v^k=r^k+s^k$, where $r^k\in \partial_{\varepsilon_k} f(y^k)$ and $s^k\in \partial g(y^k)$.
Using the assumption that $g$ is $\mu$-strongly convex and the definition of the $\varepsilon$-subdifferential of
$f$ we obtain, for all $x\in \HH$,
\begin{align*}
 & f(x)\geq f(y^k)+\inner{r^k}{x-y^k}-\varepsilon_k,\\
 & g(x)\geq g(y^k)+\inner{s^k}{x-y^k}+\dfrac{\mu}{2}\norm{x-y^k}^2,
\end{align*}
which in turn combined with the definition of $h(\cdot)$ in \eqref{eq:co}, 
the fact that $v^k=r^k+s^k$ and \eqref{eq:def.gammak} yields the desired result.

\mgap

(b) Let us proceed by induction on $k\geq 0$. The result is trivially true for $k=0$ (since $A_0\Gamma_0=0$).
Assume now that it is true for some $k\geq 0$, i.e., assume that 
$x^k=\arg\min_x\{A_k\Gamma_k(x)+\frac{1}{2}\norm{x-x^0}^2\}$.
Using the latter identity, \eqref{eq:der.Gamma}--\eqref{eq:def.betak} and Taylor's theorem we find
\begin{align}
\nonumber
A_{k+1}\Gamma_{k+1}(x)+\dfrac{1}{2}\norm{x-x^0}^2
      &=A_{k}\Gamma_{k}(x)+\dfrac{1}{2}\norm{x-x^0}^2 + a_{k+1}\gamma_{k+1}(x)\\
      \label{eq:kaua}
      &=\beta_k+\left(\dfrac{1+\mu A_k}{2}\right)\norm{x-x^k}^2+ a_{k+1}\gamma_{k+1}(x).
\end{align}
From the definition of $\gamma_{k+1}(\cdot)$ (see \eqref{eq:def.gammak}) and some simple calculus one can check that
$x^{k+1}$ as in \eqref{eq:alg_xne} is exactly the (unique) minimizer of 
$x\mapsto \left(\frac{1+\mu A_k}{2}\right)\norm{x-x^k}^2+ a_{k+1}\gamma_{k+1}(x)$. Hence, from this fact and
\eqref{eq:kaua} we obtain that 
$x^{k+1}=\arg\min_{x\in \HH}\{A_{k+1}\Gamma_{k+1}(x)+\frac{1}{2}\norm{x-x^0}^2\}$, completing the induction argument.
\end{proof}

\mgap

\begin{lemma}
 \label{lm:gauss}
Consider the sequences evolved by \emph{Algorithm \ref{alg:main}}.
The following holds for all $x\in \HH$:
\begin{itemize}
\item[\emph{(a)}] For all $k\geq 0$,
\begin{align*}
A_k \Gamma_k(x)+\dfrac{1}{2}\norm{x-x^0}^2=\beta_k+ \left(\dfrac{1+\mu A_k}{2}\right)\norm{x-x^k}^2.
\end{align*}
\item[\emph{(b)}] For all $k\geq 0$,
\begin{align*}
A_{k+1} \Gamma_{k+1}(x)+\dfrac{1}{2}\norm{x-x^0}^2=\beta_k+
\left(\dfrac{1+\mu A_k}{2}\right)\norm{x-x^k}^2 + a_{k+1}\gamma_{k+1}(x).
\end{align*}
\item[\emph{(c)}] For all $k\geq 0$,
\begin{align*}
 A_k h(y^k)+
A_{k+1} \Gamma_{k+1}(x)+\dfrac{1}{2}\norm{x-x^0}^2\geq \beta_k+
\left(\dfrac{1+\mu A_k}{2}\right)\norm{x-x^k}^2 + a_{k+1}\gamma_{k+1}(x)
+ A_k \gamma_{k+1}(y^k).
\end{align*}
\end{itemize}
\end{lemma}
\begin{proof}
(a) First note that the result is trivial for $k=0$, since $\beta_0=A_0=0$ and $\Gamma_0=0$.
Now note that in view of \eqref{eq:der.Gamma} we obtain, for $k\geq 1$,
\[
\nabla^2 \left(A_k \Gamma_k(\cdot)+\dfrac{1}{2}\norm{\cdot-x^0}^2\right)(x)= 1+\mu A_k.
\]
Using the latter identity, Lemma \ref{lm:initial}(b), \eqref{eq:def.betak} and Taylor's theorem we find
\begin{align*}
A_k \Gamma_k(x)+\dfrac{1}{2}\norm{x-x^0}^2&=
        \underbrace{A_k \Gamma_k(x^k)+\dfrac{1}{2}\norm{x^k-x^0}^2}_{\beta_k}+
        \dfrac{1}{2}\inner{(1+\mu A_k)(x-x^k)}{x-x^k}\\
        &=\beta_k+\left(\dfrac{1+\mu A_k}{2}\right)\norm{x-x^k}^2.
\end{align*}

(b) From \eqref{eq:rec.ak} and item  (a), we obtain, for all $k\geq 0$,
\begin{align*}
A_{k+1} \Gamma_{k+1}(x)+\dfrac{1}{2}\norm{x-x^0}^2&=A_k\Gamma_k(x)+\dfrac{1}{2}\norm{x-x^0}^2
  +a_{k+1}\gamma_{k+1}(x)\\
  &=\beta_k+\left(\dfrac{1+\mu A_k}{2}\right)\norm{x-x^k}^2+a_{k+1}\gamma_{k+1}(x).
\end{align*}

(c) From (b) and Lemma \ref{lm:initial}(a) with $k=k+1$ and  $x=y^k$,

\begin{align*}
A_k h(y^k)+A_{k+1} \Gamma_{k+1}(x)+\dfrac{1}{2}\norm{x-x^0}^2&=\beta_k+
\left(\dfrac{1+\mu A_k}{2}\right)\norm{x-x^k}^2 + a_{k+1}\gamma_{k+1}(x)+A_k h(y^k)\\
&\geq \beta_k+
\left(\dfrac{1+\mu A_k}{2}\right)\norm{x-x^k}^2 + a_{k+1}\gamma_{k+1}(x)+A_k \gamma_{k+1}(y^k).
\end{align*}
\end{proof}

\mgap

\begin{lemma}
 \label{lm:newton}
 Consider the sequences evolved by \emph{Algorithm \ref{alg:main}}. The following holds:
\begin{itemize}
\item[\emph{(a)}] For all $k\geq 0$ and $x\in \HH$,
\begin{align*}
a_{k+1}\gamma_{k+1}(x)+A_k \gamma_{k+1}(y^k)=
A_{k+1}\gamma_{k+1}(\widetilde x)
+ \left(\dfrac{\mu\, a_{k+1}A_k}{2 A_{k+1}}\right)\norm{x-y^k}^2,
\end{align*}
where
\begin{align}
 \label{eq:def.xtilde}
\widetilde x:= \dfrac{a_{k+1}}{A_{k+1}}x+\dfrac{A_k}{A_{k+1}}y^k.
\end{align}
\item[\emph{(b)}] For all $k\geq 0$ and $x\in \HH$,
\begin{align*}
A_k h(y^k)+
A_{k+1} \Gamma_{k+1}(x)+\dfrac{1}{2}\norm{x-x^0}^2\geq \beta_k+
A_{k+1}\Big[\gamma_{k+1}(\widetilde x)+\Delta_k\Big],
\end{align*}
where, for all $k\geq 0$, $\widetilde x$ is as in \eqref{eq:def.xtilde} and
\begin{align}
\label{eq:def.deltak}
&\Delta_k:=\left(\dfrac{(1+\mu A_k)A_{k+1}}{2a_{k+1}^2}\right)\norm{\widetilde x-z^k}^2 
+\left(\dfrac{\mu A_k}{2a_{k+1}}\right)\norm{\widetilde x-y^k}^2,\\
 \label{eq:def.xktilde}
&z^k:= \dfrac{a_{k+1}}{A_{k+1}}x^k+\dfrac{A_k}{A_{k+1}}y^k.
\end{align}
\item[\emph{(c)}] For all $k\geq 0$,
\begin{align}
\label{eq:kaua02}
\Delta_k&=\dfrac{1}{2\lambda_{k+1}}\left[\norm{\widetilde x-\widetilde x^k}^2
+\left(\dfrac{\mu(1+\mu A_k)\lambda_{k+1}^2A_k}{a_{k+1}A_{k+1}}\right)\norm{x^k-y^k}^2
\right],
\end{align}
where $\widetilde x$ is as in \eqref{eq:def.xtilde}.
\item[\emph{(d)}] For all $k\geq 0$ and $x\in \HH$,
\begin{align*}
A_k h(y^k)+
A_{k+1} \Gamma_{k+1}(x)+\dfrac{1}{2}\norm{x-x^0}^2\geq \beta_k+
A_{k+1}h(y^{k+1})&+\left(\dfrac{1-\sigma^2}{2}\right)
\left(\dfrac{A_{k+1}}{\lambda_{k+1}}\norm{y^{k+1}-\widetilde x^k}^2\right)\\[2mm]
&+\left(\dfrac{\mu(1+\mu A_k)\lambda_{k+1}A_k}{2a_{k+1}}\right)\norm{x^k-y^k}^2.
\end{align*}

\end{itemize}
\end{lemma}
\begin{proof}
(a) First recall that (see \eqref{eq:def.gammak})
\begin{align}
 \label{eq:recall.gamma}
\gamma_{k+1}(x)=\underbrace{h(y^{k+1})+
\inner{v^{k+1}}{x-y^{k+1}}-\varepsilon_{k+1}}_{\ell_{k+1}(x)}+\dfrac{\mu}{2}\norm{x-y^{k+1}}^2, \qquad \forall x\in \HH.
\end{align}
Let $p=\frac{a_{k+1}}{A_{k+1}}$, $q=\frac{A_k}{A_{k+1}}$ and note that $p,q\geq 0$, $p+q=1$
and $\widetilde x=px+qy^k$. Since 
$\ell_{k+1}(\cdot)$ is affine, we find
\begin{align}
\nonumber
\ell_{k+1}(\widetilde x)=
\ell_{k+1}(px+qy^k)&=p\ell_{k+1}(x)+q\ell_{k+1}(y^k)\\
\label{eq:ee.2}
   &=\dfrac{1}{A_{k+1}}\left[a_{k+1}\ell_{k+1}(x)+A_k\ell_{k+1}(y^k)\right].
\end{align}
On the other hand, using the well-know identity $\norm{pz+qw}^2=p\norm{z}^2+q\norm{w}^2-pq\norm{z-w}^2$, for all $z,w\in \HH$, we 
also find
\begin{align}
\nonumber
\norm{\widetilde x-y^{k+1}}^2&=\norm{p(x-y^{k+1})+q(y^k-y^{k+1})}^2\\
\nonumber
     &=p\norm{x-y^{k+1}}^2+q\norm{y^k-y^{k+1}}^2-pq\norm{x-y^k}^2\\
     \label{eq:norm.2}
     &=\dfrac{1}{A_{k+1}}\left[a_{k+1}\norm{x-y^{k+1}}^2+A_k\norm{y^k-y^{k+1}}^2-
     \left(\dfrac{a_{k+1}A_k}{A_{k+1}}\right)\norm{x-y^k}^2\right].
\end{align}

Combining \eqref{eq:recall.gamma}--\eqref{eq:norm.2}, we then obtain
\begin{align*}
\gamma_{k+1}(\widetilde x)&=\ell_{k+1}(\widetilde x)+\dfrac{\mu}{2}\left\|\widetilde x-y^{k+1}\right\|^2\\
     &=\dfrac{1}{A_{k+1}}
     \left[a_{k+1}\left(\ell_{k+1}(x)+\dfrac{\mu}{2}\norm{x-y^{k+1}}^2\right)
     +A_k\left(\ell_{k+1}(y^k)+\dfrac{\mu}{2}\norm{y^k-y^{k+1}}^2\right)
     -\left(\dfrac{\mu a_{k+1}A_k}{2A_{k+1}}\right)\norm{x-y^k}^2
     \right]\\
     &=\dfrac{1}{A_{k+1}}\left[a_{k+1}\gamma_{k+1}(x)+A_k\gamma_{k+1}(y^k)-
     \left(\dfrac{\mu a_{k+1}A_k}{2A_{k+1}}\right)\norm{x-y^k}^2\right],
\end{align*}
which is clearly equivalent to the desired identity.

\mgap

(b) First note that in view of \eqref{eq:def.xtilde} and \eqref{eq:def.xktilde} we have 
 $\widetilde x-z^k=\frac{a_{k+1}}{A_{k+1}}(x-x^k)$ and, analogously, we also have tilde $x-y^k=\frac{a_{k+1}}{A_{k+1}}(x-y^k)$. Hence,
 \begin{align}
  \label{eq:mud.xtilde}
 \norm{x-x^k}^2=\dfrac{A_{k+1}^2}{a_{k+1}^2}\norm{\widetilde x-z^k}^2
 \;\;\;\mbox{and}\;\;\;
 \norm{x-y^k}^2=\dfrac{A_{k+1}^2}{a_{k+1}^2}\norm{\widetilde x-y^k}^2.
 \end{align}
 Using Lemma \ref{lm:gauss}(c) and item (a) we find
 
 \begin{align*}
 A_k h(y^k)+
A_{k+1} \Gamma_{k+1}(x)+\dfrac{1}{2}\norm{x-x^0}^2&\geq \beta_k+
 \left(\dfrac{1+\mu A_k}{2}\right)\norm{x-x^k}^2 + a_{k+1}\gamma_{k+1}(x)
+ A_k \gamma_{k+1}(y^k)\\
&= \beta_k+A_{k+1}\gamma_{k+1}(\widetilde x)\\
&\hspace{1cm}+ \left(\dfrac{1+\mu A_k}{2}\right)\norm{x-x^k}^2
  +\left(\dfrac{\mu\, a_{k+1}A_k}{2 A_{k+1}}\right)\norm{x-y^k}^2,
\end{align*}
which in turn combined with \eqref{eq:mud.xtilde} and \eqref{eq:def.deltak} finishes the proof of item (b).

\mgap

(c) First let $p=\frac{(1+\mu A_k)A_{k+1}\lambda_{k+1}}{a_{k+1}^2}$,
$q=\dfrac{\mu A_k\lambda_{k+1}}{a_{k+1}}$ and note that $p,q\geq 0$ and, in view of \eqref{eq:cond.lamb}, 
$p+q=1$. From \eqref{eq:def.deltak} and the above definitions of $p$ and $q$, we obtain
\begin{align}
\nonumber
\Delta_k&=
\left(\dfrac{(1+\mu A_k)A_{k+1}}{2a_{k+1}^2}\right)\norm{\widetilde x-z^k}^2 
+\left(\dfrac{\mu A_k}{2a_{k+1}}\right)\norm{\widetilde x-y^k}^2\\
\nonumber
&= \dfrac{1}{2\lambda_{k+1}}
\left[p\norm{\widetilde x-z^k}^2 
+q\norm{\widetilde x-y^k}^2\right]\\
\label{eq:first.part}
 &= \dfrac{1}{2\lambda_{k+1}}\left[\norm{\widetilde x-(p z^k+qy^k)}^2+pq\norm{y^k-z^k}^2\right],
\end{align}
where we also used the well-known identity $p\norm{z}^2+q\norm{w}^2=\norm{pz+qw}^2+pq\norm{z-w}^2$, for $z,w\in \HH$.

Using \eqref{eq:def.xktilde}, the definitions of $p,q$, the fact that $p+q=1$, \eqref{eq:alg_xtil} and \eqref{eq:alg_A}, and some simple computations, we find
\begin{align}
\nonumber
p z^k+qy^k&=(1-q)\left(\dfrac{a_{k+1}}{A_{k+1}}x^k+\dfrac{A_k}{A_{k+1}}y^k\right)+q y^k\\
\nonumber
&=(1-q) \dfrac{a_{k+1}}{A_{k+1}}x^k+\left(\dfrac{A_k}{A_{k+1}}+q\left(1-\dfrac{A_k}{A_{k+1}}\right)\right)y^k\\
\nonumber
&=(1-q) \dfrac{a_{k+1}}{A_{k+1}}x^k+\left(\dfrac{A_k}{A_{k+1}}+q \dfrac{a_{k+1}}{A_{k+1}}\right)y^k\\
\nonumber
&=\left(1-\dfrac{\mu A_k\lambda_{k+1}}{a_{k+1}}\right) \dfrac{a_{k+1}}{A_{k+1}}x^k+\left(\dfrac{A_k}{A_{k+1}}+\left(\dfrac{\mu A_k\lambda_{k+1}}{a_{k+1}}\right)\dfrac{a_{k+1}}{A_{k+1}}\right)y^k\\
\nonumber
&=\left(\dfrac{a_{k+1}-\mu A_k \lambda_{k+1}}{A_{k+1}}\right)x^k+
 \left(\dfrac{A_k+\mu A_k\lambda_{k+1}}{A_{k+1}}\right)y^k\\
 \label{eq:ide.txk}
 &=\widetilde x^k.
\end{align}
On the other hand, using again \eqref{eq:def.xktilde} and the definitions of $p,q$, we also obtain
\begin{align}
\nonumber
pq\norm{y^k-z^k}^2&=\left(\frac{(1+\mu A_k)A_{k+1}\lambda_{k+1}}{a_{k+1}^2}\right)
\left(\dfrac{\mu A_k\lambda_{k+1}}{a_{k+1}}\right)\dfrac{a_{k+1}^2}{A_{k+1}^2}\norm{x^k-y^k}^2\\
\label{eq:part.pq}
&=\left(\dfrac{\mu(1+\mu A_k)\lambda_{k+1}^2A_k}{a_{k+1}A_{k+1}}\right)\norm{x^k-y^k}^2.
\end{align}
The desired result now follows directly from \eqref{eq:first.part}, \eqref{eq:ide.txk} and \eqref{eq:part.pq}.

\mgap

(d) From items (b) and (c), 
\begin{align}
 \label{eq:broyden}
\nonumber
A_k h(y^k)+
A_{k+1} \Gamma_{k+1}(x)+\dfrac{1}{2}\norm{x-x^0}^2\geq \beta_k&+
A_{k+1}\Big[\gamma_{k+1}(\widetilde x)+\dfrac{1}{2\lambda_{k+1}} \norm{\widetilde x-\widetilde x^k}^2\Big]\\[2mm]
&+\left(\dfrac{\mu(1+\mu A_k)\lambda_{k+1}A_k}{2a_{k+1}}\right)\norm{x^k-y^k}^2.
\end{align}
From \eqref{eq:def.gammak},
\begin{align}
\nonumber
\gamma_{k+1}(\widetilde x)+\dfrac{1}{2\lambda_{k+1}} \norm{\widetilde x-\widetilde x^k}^2&=
h(y^{k+1})\\
\label{eq:lete}
& \hspace{0.7cm}+\underbrace{\inner{v^{k+1}}{\widetilde x-y^{k+1}}+\dfrac{\mu}{2}\norm{\widetilde x-y^{k+1}}^2-
\varepsilon_{k+1}+\dfrac{1}{2\lambda_{k+1}} \norm{\widetilde x-\widetilde x^k}^2}_{=:q_{k+1}(\widetilde x)}.
\end{align}
On the other hand, from Lemma \ref{lm:wolfe}(c) applied to $q_{k+1}(\cdot)$ and \eqref{eq:alg_err},
\begin{align*}
q_{k+1}(\widetilde x)\geq \left(\dfrac{1-\sigma^2}{2\lambda_{k+1}}\right)\norm{y^{k+1}-\widetilde x^k}^2,
\end{align*}
which in turn combined with \eqref{eq:lete} gives
\begin{align*}
 \gamma_{k+1}(\widetilde x)+\dfrac{1}{2\lambda_{k+1}} \norm{\widetilde x-\widetilde x^k}^2
 \geq h(y^{k+1})+\left(\dfrac{1-\sigma^2}{2\lambda_{k+1}}\right)\norm{y^{k+1}-\widetilde x^k}^2.
\end{align*}
The desired result now follows by the substitution of the latter inequality in \eqref{eq:broyden}.
\end{proof}

\mgap

Next is our first result on the iteration-complexity of Algorithm \ref{alg:main}. 
Item (b) follows trivially from item (a), 
which will be derived from Lemmas \ref{lm:initial}, \ref{lm:gauss} and \ref{lm:newton}. 
The main results on the iteration-complexity of Algorithm \ref{alg:main} will then be presented in 
Theorem \ref{th:main_comp} below.

\mgap

\begin{proposition}
 \label{pr:cambio}
 Consider the sequences evolved by \emph{Algorithm \ref{alg:main}}, let $x^*$ denote the (unique) solution of \eqref{eq:co} and let 
 \begin{align}
  \label{eq:def.d0}
  d_0:=\norm{x^*-x^0}.
 \end{align}
 The following holds:
 \begin{itemize}
 \item[\emph{(a)}] For all $k\geq 1$ and $x\in \HH$,
\begin{align*}
 & A_{k}\left[h(y^{k})-h(x)\right]+\left(\dfrac{1-\sigma^2}{2}\right)\sum_{j=1}^{k}\,
\dfrac{A_j}{\lambda_j}\norm{y^j-\widetilde x^{\,j-1}}^2\\
& \hspace{1cm}
+\sum_{j=1}^{k}\,\left(\dfrac{\mu(1+\mu A_{j-1})\lambda_{j}A_{j-1}}{2a_{j}}\right)\norm{x^{j-1}-y^{j-1}}^2
+ \left(\dfrac{1+\mu A_{k}}{2}\right)\norm{x-x^{k}}^2\leq \dfrac{1}{2}\norm{x-x^0}^2.
\end{align*}
 \item[\emph{(b)}] If $\sigma<1$, for all $k\geq 1$,
 \begin{align}
  \label{eq:sting3}
  \sum_{j=1}^{k}\,
\dfrac{A_j}{\lambda_j}\norm{y^j-\widetilde x^{\,j-1}}^2\leq \dfrac{d_0^2}{1-\sigma^2},
\qquad \forall k\geq 1.
 \end{align}
 \end{itemize}
\end{proposition}
\begin{proof}
(a) From Lemma \ref{lm:newton}(d) and the definition of $\beta_{k+1}$ -- see \eqref{eq:def.betak} -- we obtain,
for all $k\geq 0$,
\begin{align*}
A_k h(y^k)+\beta_{k+1}\geq \beta_k+
A_{k+1}h(y^{k+1})&+\left(\dfrac{1-\sigma^2}{2}\right)
\left(\dfrac{A_{k+1}}{\lambda_{k+1}}\norm{y^{k+1}-\widetilde x^k}^2\right)\\[2mm]
&+\left(\dfrac{\mu(1+\mu A_k)\lambda_{k+1}A_k}{2a_{k+1}}\right)\norm{x^k-y^k}^2,
\end{align*}
and so, for all $k\geq 0$, 
\begin{align*}
\underbrace{\sum_{j=0}^{k}\,\left[\beta_{j+1}-\beta_j\right]}_{\beta_{k+1}-\beta_0}
\geq \underbrace{\sum_{j=0}^{k}\,\left[A_{j+1}h(y^{j+1})-A_jh(y^j)\right]}_{A_{k+1}h(y^{k+1})-A_0h(y^0)}
&+\left(\dfrac{1-\sigma^2}{2}\right)\sum_{j=0}^{k}\,
\dfrac{A_{j+1}}{\lambda_{j+1}}\norm{y^{j+1}-\widetilde x^{j}}^2\\
&+\sum_{j=0}^{k}\,\left(\dfrac{\mu(1+\mu A_j)\lambda_{j+1}A_j}{2a_{j+1}}\right)\norm{x^j-y^j}^2,
\end{align*}
which, since $\beta_0=A_0=0$, yields, for all $k\geq 0$,
\begin{align*}
\beta_{k+1}\geq A_{k+1}h(y^{k+1})+\left(\dfrac{1-\sigma^2}{2}\right)\sum_{j=1}^{k+1}\,
\dfrac{A_j}{\lambda_j}\norm{y^j-\widetilde x^{\,j-1}}^2
+\sum_{j=1}^{k+1}\,\left(\dfrac{\mu(1+\mu A_{j-1})\lambda_{j}A_{j-1}}{2a_{j}}\right)\norm{x^{j-1}-y^{j-1}}^2.
\end{align*}

By adding $\left(\frac{1+\mu A_{k+1}}{2}\right)\norm{x-x^{k+1}}^2$ in both sides of the latter inequality, we obtain,
for all $k\geq 0$,
\begin{align*}
\beta_{k+1}+\left(\dfrac{1+\mu A_{k+1}}{2}\right)\norm{x-x^{k+1}}^2&\geq A_{k+1}h(y^{k+1})+\left(\dfrac{1-\sigma^2}{2}\right)\sum_{j=1}^{k+1}\,
\dfrac{A_j}{\lambda_j}\norm{y^j-\widetilde x^{\,j-1}}^2\\
&+\sum_{j=1}^{k+1}\,\left(\dfrac{\mu(1+\mu A_{j-1})\lambda_{j}A_{j-1}}{2a_{j}}\right)\norm{x^{j-1}-y^{j-1}}^2\\
&+\left(\dfrac{1+\mu A_{k+1}}{2}\right)\norm{x-x^{k+1}}^2.
\end{align*}
Using Lemma \ref{lm:gauss}(a) we then find, for all $k\geq 0$,
\begin{align}
 \label{eq:aint}
\nonumber
 A_{k+1}\Gamma_{k+1}(x)+\dfrac{1}{2}\norm{x-x^0}^2&\geq A_{k+1}h(y^{k+1})+\left(\dfrac{1-\sigma^2}{2}\right)\sum_{j=1}^{k+1}\,
\dfrac{A_j}{\lambda_j}\norm{y^j-\widetilde x^{\,j-1}}^2\\
\nonumber
&+\sum_{j=1}^{k+1}\,\left(\dfrac{\mu(1+\mu A_{j-1})\lambda_{j}A_{j-1}}{2a_{j}}\right)\norm{x^{j-1}-y^{j-1}}^2\\
&+ \left(\dfrac{1+\mu A_{k+1}}{2}\right)\norm{x-x^{k+1}}^2.
\end{align}

Note now that from \eqref{eq:def.Gammak} and Lemma \ref{lm:initial}(a) we obtain, for all $k\geq 0$,
\[
 A_{k+1}\Gamma_{k+1}(x)=\sum_{j=1}^{k+1}\,a_j\gamma_j(x)\leq A_{k+1}h(x),
\]
which combined with \eqref{eq:aint} yields, for all $k\geq 1$,
\begin{align*}
 \dfrac{1}{2}\norm{x-x^0}^2&\geq A_{k}\left[h(y^{k})-h(x)\right]+\left(\dfrac{1-\sigma^2}{2}\right)\sum_{j=1}^{k}\,
\dfrac{A_j}{\lambda_j}\norm{y^j-\widetilde x^{\,j-1}}^2\\
& \hspace{1cm}
+\sum_{j=1}^{k}\,\left(\dfrac{\mu(1+\mu A_{j-1})\lambda_{j}A_{j-1}}{2a_{j}}\right)\norm{x^{j-1}-y^{j-1}}^2
+ \left(\dfrac{1+\mu A_{k}}{2}\right)\norm{x-x^{k}}^2.
\end{align*}

\mgap

(b) This follows trivially from item (a) and \eqref{eq:def.d0}.
\end{proof}

\mgap

\begin{lemma}
\label{lm:standard}
For all $k\geq 0$,
\begin{align}
 \label{eq:standard2}
\left(1-\sigma\sqrt{1+\lambda_{k+1}\mu}\right)\norm{y^{k+1}-\widetilde x^k}\leq
\norm{\lambda_{k+1}v^{k+1}}\leq 
\left(1+\sigma\sqrt{1+\lambda_{k+1}\mu}\right)\norm{y^{k+1}-\widetilde x^k}.
\end{align}
\end{lemma}
\begin{proof}
 The proof follows from the inequality in \eqref{eq:alg_err}, the fact that $\varepsilon_{k+1}\geq 0$ and a simple argument based on the triangle inequality.
\end{proof}

\mgap

Since, under mild regularity assumptions on $f$ and $g$, problem \eqref{eq:co} is equivalent to the inclusion
\begin{align}
 \label{eq:mipfg}
0\in \partial f(x)+\partial g(x),
\end{align}
it is natural to attempt to evaluate the residuals produced by Algorithm \ref{alg:main}  in the light of \eqref{eq:mipfg}, and this is exactly what Theorem \ref{th:main_comp}(b) is about. 
Note that if we set $v^{k+1}=0$ and $\varepsilon_{k+1}=0$ in \eqref{eq:jeffblack}, then it follows that $x:=y^{k+1}$
satisfies the inclusion \eqref{eq:mipfg}.

\mgap

As we mentioned before, Theorem \ref{th:main_comp} below is our main result on the iteration-complexity of
Algorithm \ref{alg:main}. 

\mgap

\begin{theorem}[{\bf Convergence rates for Algorithm \ref{alg:main}}]
 \label{th:main_comp}
 Consider the sequences evolved by \emph{Algorithm \ref{alg:main}}, let $x^*$ be the (unique) solution of \eqref{eq:co} and let $d_0$ be as in
 \eqref{eq:def.d0}.
 Then, the following holds:
 \begin{itemize}
 \item[\emph{(a)}] For all $k\geq 1$,
 \begin{align*}
   h(y^{k})-h(x^*)\leq \dfrac{d_0^2}{2 A_k},\qquad \norm{x^*-y^k}^2\leq \dfrac{d_0^2}{\mu A_k},\qquad 
   \norm{x^*-x^{k}}^2\leq \dfrac{d_0^2}{1+\mu A_k}.
 \end{align*}
\item[\emph{(b)}] For all $k\geq 1$, 
\begin{align}
 \label{eq:jeffblack}
\begin{cases}
&v^{k+1}\in \partial_{\varepsilon_{k+1}} f(y^{k+1})+\partial g(y^{k+1}),\\[2mm]
&\norm{v^{k+1}}^2\leq 
\left(\dfrac{1+\sigma\sqrt{1+\mu\lambda_{k+1}}}{6^{-1/2} \lambda_{k+1}}\right)^2 \dfrac{d_0^2}{\mu A_k},\\[6mm]
&\varepsilon_{k+1}\leq \left(\dfrac{3\sigma^2}{\lambda_{k+1}}\right) \dfrac{d_0^2}{\mu A_k}.
\end{cases}
\end{align}
 \end{itemize}
\end{theorem}
\begin{proof}
(a) Note that the bounds on $h(y^k)-h(x^*)$ and $\norm{x^*-x^k}^2$ follow directly from Proposition \ref{pr:cambio}(a)
with $x=x^*$ and \eqref{eq:def.d0}. Now, since $h(\cdot)$ is $\mu$-strongly convex and $0\in \partial h(x^*)$, one can use the 
inequality (see, e.g., \cite[Proposition 6(c)]{roc-mon.sjco76}) 
$h(x)\geq h(x^*)+\frac{\mu}{2}\norm{x-x^*}^2$, for all $x\in \HH$, with $x=y^k$ and the bound on
$h(y^k)-h(x^*)$ to conclude that
$\norm{y^k-x^*}^2\leq \frac{2}{\mu}\left(h(y^k)-h(x^*)\right)\leq \frac{d_0^2}{\mu A_k}$.

(b) First, note that the inclusion in \eqref{eq:jeffblack} follows from the inclusion in \eqref{eq:alg_err}. 
Since we will use the second inequality in \eqref{eq:standard2} to prove the 
inequality for $\norm{v^{k+1}}^2$, it follows that we first have to bound the term $\norm{y^{k+1}-\widetilde x^{k}}^2$.
To this end, note that from the second inequality in item (a) with $k=k+1$ and the fact that $A_{k+1}\geq A_k$,
\begin{align}
 \label{eq:bjovi}
\nonumber
 \norm{y^{k+1}-\widetilde x^{k}}^2&\leq 2\left(\norm{x^*-y^{k+1}}^2 +\norm{\widetilde x^{k}-x^*}^2\right)\\
   & \leq 2\left(\dfrac{d_0^2}{\mu A_k}+\norm{\widetilde x^{k}-x^*}^2\right).
\end{align}
We now have to bound the second term in \eqref{eq:bjovi}. Since, from \eqref{eq:alg_xtil}, $\widetilde x^k$ is a convex combination of $x^k$ and $y^k$, it follows that
\begin{align}
 \label{eq:bjovi2}
\nonumber
\norm{\widetilde x^{k}-x^*}^2&\leq \norm{x^*-x^k}^2+\norm{x^*-y^k}^2\\
\nonumber
   &\leq \dfrac{d_0^2}{1+\mu A_k} + \dfrac{d_0^2}{\mu A_k}\\
   &\leq \dfrac{2d_0^2}{\mu A_k},
\end{align}
where in the second inequality we used the second and third inequalities in item (a). 
Now using \eqref{eq:bjovi} and \eqref{eq:bjovi2}, we find
\begin{align}
 \label{eq:bjovi3}
 \norm{y^{k+1}-\widetilde x^k}^2\leq 6\dfrac{d_0^2}{\mu A_k}.
\end{align}
To finish the proof of (b), note that using \eqref{eq:bjovi3}, we obtain the desired bounds on $\norm{v^{k+1}}^2$ and $\varepsilon_{k+1}$ as a consequence of the second inequality in \eqref{eq:standard2} and the fact that
$2\lambda_{k+1}\varepsilon_{k+1}\leq \sigma^2\norm{y^{k+1}-\widetilde x^k}^2$ (see \eqref{eq:alg_err}), respectively. 
\end{proof}

\mgap

Next result is motivated by the fact that the rate of convergence of Algorithm \ref{alg:main} presented in 
Theorem \ref{th:main_comp} is given in terms of the sequence $\{A_k\}$.
We also mention that the proof of Lemma \ref{lm:bbking} follows the same outline of an argument given
in p. 13 of \cite{bach-preprint21}.

\mgap

\begin{lemma}
 \label{lm:bbking}
The following holds:
\begin{itemize}
\item[\emph{(a)}] For all $k\geq 1$,
\begin{align}
 \label{eq:stingz}
A_{k+1} \geq
\lambda_1 \prod_{j=2}^{k+1}
\left(\dfrac{1}{1-\sqrt{\dfrac{\mu\lambda_{j}}{1+\mu\lambda_{j}}}}\right).
\end{align}
\item[\emph{(b)}] For all $k\geq 1$,
\begin{align}
 \label{eq:sting2}
A_{k+1} \geq \lambda_1
 \prod_{j=2}^{k+1}\left(1+2\mu\lambda_{j}\right).
\end{align}
\end{itemize}
\end{lemma}
\begin{proof}
(a) From \eqref{eq:alg_a},
\begin{align*}
a_{k+1}&=\dfrac{(1+2\mu A_k)\lambda_{k+1}+\sqrt{(1+2\mu A_k)^2\lambda_{k+1}^2
+4(1+\mu A_k)A_k\lambda_{k+1}}}{2}\\
   &\geq \dfrac{(2\mu A_k)\lambda_{k+1}+\sqrt{(2\mu A_k)^2\lambda_{k+1}^2
+4(\mu A_k)A_k\lambda_{k+1}}}{2}\\
&=\dfrac{(2\mu A_k)\lambda_{k+1}+2A_k\sqrt{\mu^2\lambda_{k+1}^2
+\mu \lambda_{k+1}}}{2}\\
&= A_k\left[\mu\lambda_{k+1}+\sqrt{\mu\lambda_{k+1}(1+\mu\lambda_{k+1})}\right].
\end{align*}
Hence, from \eqref{eq:alg_A},
\begin{align}
\nonumber
 A_{k+1}&=A_k+a_{k+1}\\
 \nonumber
              & \geq A_k+ A_k\left[\mu\lambda_{k+1}+\sqrt{\mu\lambda_{k+1}(1+\mu\lambda_{k+1})}\right]\\
  \label{eq:blackwell02}
              &= A_k\left[1+\mu\lambda_{k+1}+\sqrt{\mu\lambda_{k+1}(1+\mu\lambda_{k+1})}\right]\\[2mm]
              \label{eq:sting}
              &= A_k\left(\dfrac{1}{1-\sqrt{\dfrac{\mu\lambda_{k+1}}{1+\mu\lambda_{k+1}}}}\right),   
\end{align}
where in the last equality we used the identity $1/\left(1-\sqrt{\frac{x}{1+x}}\right)=1+x+\sqrt{x(1+x)}$ with 
$x=\mu\lambda_{k+1}$.
Note now that \eqref{eq:stingz} follows directly from \eqref{eq:sting} and the fact that  $A_1=\lambda_1$ -- see \eqref{eq:A1a1}.

\mgap

(b) Using \eqref{eq:blackwell02}, the fact that 
$\sqrt{\mu\lambda_{k+1}(1+\mu\lambda_{k+1})}\geq \mu\lambda_{k+1}$ and a similar reasoning to the proof of item (a), we obtain that \eqref{eq:sting2} holds for all $k\geq 1$.
\end{proof}

\mgap

Next is a corollary of Lemma \ref{lm:bbking}(a) for the special case that the sequence 
$\{\lambda_k\}$ is bounded away from zero. Lemma \ref{lm:bbking}(b) will be useful later in Section \ref{sec:ls}.

\mgap

\begin{corollary}
 \label{cor:bounded}
Assume that $\lambda_k\geq \underline{\lambda}>0$, for all $k\geq 1$, and define $\alpha\in (0,1)$ as
\begin{align}
 \label{eq:def.alpha}
 \alpha:=\sqrt{\dfrac{\mu\underline{\lambda}}{1+\mu\underline{\lambda}}}.
\end{align}
Then, for all $k\geq 1$,
\begin{align}
 \label{eq:borwein}
   A_k\geq \underline{\lambda} \left(\dfrac{1}{1-\alpha}\right)^{k-1}.
\end{align}
\end{corollary}
\begin{proof}
Using the fact that the scalar function $(0,\infty)\ni t \mapsto \frac{\mu t}{1+\mu t}\in (0,1)$ is increasing, the
assumption $\lambda_k\geq \underline{\lambda}>0$, for all $k\geq 1$, and \eqref{eq:def.alpha}, we find
\begin{align*}
\dfrac{1}{1-\sqrt{\dfrac{\mu\lambda_{j}}{1+\mu\lambda_{j}}}}\geq \dfrac{1}{1-\alpha},\qquad \forall j\geq 1.
\end{align*}
Hence, from Lemma \ref{lm:bbking}(a) and the assumption $\lambda_k\geq \underline{\lambda}$ with $k=1$ we obtain
$A_{k+1}\geq \underline{\lambda} \left(\dfrac{1}{1-\alpha}\right)^k$, for all $k\geq 1$, which is
clearly equivalent to $A_{k}\geq \underline{\lambda} \left(\dfrac{1}{1-\alpha}\right)^{k-1}$ for all $k\geq 2$.
To finish the proof of \eqref{eq:borwein}, note that the latter inequality holds trivially for $k=1$ (because $A_1=\lambda_1$ and
$\lambda_1\geq \underline{\lambda}$).
\end{proof}

\mgap

Next we present convergence rate results for Algorithm \ref{alg:main} under the assumption that
$\{\lambda_k\}$ is bounded away from zero.

\mgap

\begin{theorem}[{\bf Convergence rates for Algorithm \ref{alg:main} with $\{\lambda_k\}$ bounded below}]
 \label{th:main_comp03}
 Consider the sequences evolved by \emph{Algorithm \ref{alg:main}} and assume that 
 $\lambda_k\geq \underline{\lambda}>0$ for all $k\geq 1$. Let $x^*$ be the (unique) solution of \eqref{eq:co}, let $d_0$ be as in  \eqref{eq:def.d0} and let $\alpha\in (0,1)$ be as in \eqref{eq:def.alpha}.
 The following holds:
 \begin{itemize}
 \item[\emph{(a)}] For all $k\geq 1$,
 \begin{align*}
  & h(y^{k})-h(x^*)\leq \dfrac{d_0^2}{2\underline{\lambda}}(1-\alpha)^{k-1},\\[3mm]
  & \max\left\{\norm{x^*-y^k},\norm{x^*-x^k}\right\}\leq 
      \dfrac{d_0}{\sqrt{\mu\underline{\lambda}}}(1-\alpha)^{(k-1)/2}.
 \end{align*}
\item[\emph{(b)}] For all $k\geq 1$, 
\begin{align*}
%
\begin{cases}
&v^{k+1}\in \partial_{\varepsilon_{k+1}} f(y^{k+1})+\partial g(y^{k+1}),\\[3mm]
&\norm{v^{k+1}}\leq 
\left(\dfrac{1+\sigma\sqrt{1+\mu\underline{\lambda}}}{6^{-1/2}\mu^{1/2}\underline{\lambda}^{3/2}}\right)
d_0\,(1-\alpha)^{(k-1)/2},\\[6mm]
&\varepsilon_{k+1}\leq \left(\dfrac{3\sigma^2 d_0^2}{\mu \underline{\lambda}^2}\right) (1-\alpha)^{k-1}.
\end{cases}
\end{align*}

 \end{itemize}
\end{theorem}
\begin{proof}
(a) This follows from Theorem \ref{th:main_comp}(a) and Corollary \ref{cor:bounded}.

\mgap

(b) The result follows from Theorem \ref{th:main_comp}(b), Corollary \ref{cor:bounded}, 
the assumption $\lambda_k\geq \underline{\lambda}$ 
and the fact that, for $t>0$, the scalar function $t\mapsto \frac{1+\sigma\sqrt{1+\mu t}}{t}$ is nonincreasing.
\end{proof}


\section{A (high-order) large-step A-HPE algorithm for strongly convex problems}
\label{sec:ls}

In this section, we also consider problem \eqref{eq:co}, i.e., $\min_{x\in \HH}\left\{h(x):=f(x)+g(x)\right\}$, where the same assumptions as in Section \ref{sec:alg} are assumed to hold on $h$, $f$ and $g$.

For solving \eqref{eq:co}, we propose and study the iteration-complexity of a variant (Algorithm \ref{alg:ls}) of the large-step A-HPE algorithm of Monteiro and Svaiter~\cite{mon.sva-acc.siam13}, with a high-order large-step condition specially tailored for strongly convex objectives . 
Applications of this general framework to high-order tensor methods
will be given in Section \ref{sec:second}. 
The main results on convergence rates for Algorithm \ref{alg:ls} are presented in Theorem \ref{th:main_comp02} below.

\mgap

%
\noindent
\fbox{
\begin{minipage}[h]{6.6 in}
\begin{algorithm}
\label{alg:ls}
{\bf A variant of the large-step A-HPE algorithm for (the strongly convex) problem \eqref{eq:co}}
\end{algorithm}
\begin{itemize}
\item [0)] 
Choose $x^0,y^0\in \HH$, $\sigma\in [0,1)$, $p\geq 2$ and $\theta>0$; let $A_0=0$ and set $k=0$.
\item [1)] Compute $\lambda_{k+1}>0$ and $(y^{k+1},v^{k+1},\varepsilon_{k+1})\in \HH\times \HH\times \R_{+}$ such that 
\begin{align}
\begin{aligned}
 \label{eq:alg_err2}
  &v^{k+1}\in \partial_{\varepsilon_{k+1}} f(y^{k+1})+ \partial g(y^{k+1}),\\[3mm]
  &\dfrac{\norm{\lambda_{k+1}v^{k+1}+y^{k+1}-\widetilde x^k}^2}{1+\lambda_{k+1}\,\mu}
  +2\lambda_{k+1}\varepsilon_{k+1}
  \leq \sigma^2\norm{y^{k+1}-\widetilde x^k}^2,\\[3mm]
  & \lambda_{k+1}\norm{y^{k+1}-\widetilde x^k}^{p-1}\geq \theta,
\end{aligned}
\end{align}
where
\begin{align}
\label{eq:alg_xtil2}
& \widetilde x^k =  \left(\dfrac{a_{k+1}-\mu A_k\lambda_{k+1}}{A_k+a_{k+1}}\right)x^k
 + \left(\dfrac{A_k+\mu A_k\lambda_{k+1}}{A_k+a_{k+1}}\right)y^k,\\[3mm]
 \label{eq:alg_a2}
& a_{k+1}=\dfrac{(1+2\mu A_k)\lambda_{k+1}+\sqrt{(1+2\mu A_k)^2\lambda_{k+1}^2
+4(1+\mu A_k)A_k\lambda_{k+1}}}{2}. 
\end{align}

\item[2)] 
 Let
  \begin{align}
  \label{eq:alg_A2}
  &A_{k+1} = A_k + a_{k+1},\\[2mm]
  \label{eq:alg_xne2}
  & x^{k+1} = \left(\dfrac{1+\mu A_k}{1+\mu A_{k+1}}\right)x^k + 
  \left(\dfrac{\mu a_{k+1}}{1+\mu A_{k+1}}\right) y^{k+1}
  - \left(\dfrac{a_{k+1}}{1+\mu A_{k+1}}\right)v^{k+1}.
  \end{align}
\item[3)] 
Set $k=k+1$ and go to step 1.
\end{itemize}
\noindent
\end{minipage}
} 
%
%
\mgap
\mgap

We now make a few remarks concerning Algorithm \ref{alg:ls}:

\begin{itemize}
\item[(i)] By deleting 
the third inequality in \eqref{eq:alg_err2} (the high-order large-step condition), we see that  Algorithm \ref{alg:ls} is a special instance of Algorithm \ref{alg:main}. As a consequence, all results proved in Section \ref{sec:alg} for Algorithm \ref{alg:main} also hold for Algorithm \ref{alg:ls}.
\item[(ii)] We mention that Algorithm \ref{alg:ls} is a generalization of Algorithm 1 in 
\cite{jordan-controlpreprint20} to strongly convex objectives. 
The authors of the latter work proved the global rates $\mathcal{O}\left(k^{-\frac{3p+1}{2}}\right)$, 
$\mathcal{O}\left(k^{-3p}\right)$ and $\mathcal{O}\left(k^{-\frac{3p+3}{2}}\right)$ 
for (in the notation of this paper)  function values $h(y^{k+1})-h(x^*)$ and residuals $\inf_{1\leq i\leq k+1}\,\norm{v^{i}}^2$ and $\inf_{1\leq i\leq k+1}\,\varepsilon_i$, respectively. (see \cite[Theorem 4.3]{jordan-controlpreprint20}.)
\end{itemize}

In what follows we will use remark (i) following Algorithm \ref{alg:ls} to apply the results proved for 
Algorithm \ref{alg:main}  in Section \ref{sec:alg} to Algorithm \ref{alg:ls}.

The next two lemmas will be used to prove Theorem \ref{th:main_comp02} below.

\begin{lemma}
 \label{lm:emicida}
Consider the sequences evolved by \emph{Algorithm \ref{alg:ls}} and let $d_0:=\norm{x^0-x^*}$, where $x^*$ is 
the (unique) solution of \eqref{eq:co}. Then, for all $k\geq 1$,
\begin{align}
 \label{eq:furry}
 \sum_{j=1}^k\,\dfrac{A_j}{\lambda_j^\frac{p+1}{p-1}}\leq \dfrac{d_0^2}{\theta^{\frac{2}{p-1}}(1-\sigma^2)}.
\end{align}
In particular, for all $k\geq 1$,
\begin{align}
 \label{eq:furry2}
 \lambda_k\geq 
C d_0^{-\frac{2(p-1)}{p+1}},\qquad 
 C:=\lambda_1^{\frac{p-1}{p+1}}\theta^{\frac{2}{p+1}}(1-\sigma^2)^{\frac{p-1}{p+1}}.
\end{align}
\end{lemma}
\begin{proof}
Using \eqref{eq:sting3}
 and third inequality in \eqref{eq:alg_err2}, we obtain
 \begin{align*}
  \left(\sum_{j=1}^k\,\dfrac{A_j}{\lambda_j^\frac{p+1}{p-1}}\right)\theta^{\frac{2}{p-1}}
  \leq 
 \sum_{j=1}^{k}\,
\dfrac{A_j}{\lambda_j^{\frac{p+1}{p-1}}}\left(\lambda_j \norm{y^j-\widetilde x^{\,j-1}}^{p-1}\right)^{\frac{2}{p-1}}
=\sum_{j=1}^{k}\,
\dfrac{A_j}{\lambda_j}\norm{y^j-\widetilde x^{\,j-1}}^2
\leq \dfrac{d_0^2}{1-\sigma^2},
 \end{align*}
 which yields \eqref{eq:furry}. To finish the proof of the lemma, note that \eqref{eq:furry2} follows directly fom \eqref{eq:furry} and the fact that $A_k\geq \lambda_1$ for all $k\geq 1$ (see \eqref{eq:alg_A} and \eqref{eq:A1a1}).
\end{proof}

\mgap

\begin{lemma}
 \label{lm:belchior}
For all $k\geq 0$,
\begin{align}
 \label{eq:belchior}
A_{k+1}\geq 
\lambda_1\left(1+ \dfrac{2\mu C}{d_0^{\frac{2(p-1)}{p+1}}}\, k^{\,\left(\frac{p-1}{p+1}\right)}\right)^k,
\end{align}
where $C>0$ is as in \eqref{eq:furry2}.
\end{lemma}
\begin{proof}
First note that from \eqref{eq:A1a1} we have $A_1=\lambda_1$, showing that \eqref{eq:belchior} trivially holds
for $k=0$. Assume now that $k>0$.
From Lemma \ref{lm:emicida} we know, in particular, that
\begin{align*}
%
 \sum_{j=2}^{k+1}\,\dfrac{A_j}{\lambda_j^\frac{p+1}{p-1}}\leq \dfrac{d_0^{2}}{\theta^\frac{2}{p-1}(1-\sigma^2)}.
\end{align*}
Since $A_j=A_{j-1}+a_j\geq A_{j-1}\geq \dots \geq A_1$, for all $j\geq 2$, and $A_1=\lambda_1$, we then obtain
%
\begin{align*}
 \sum_{j=2}^{k+1}\,\dfrac{1}{\lambda_j^\frac{p+1}{p-1}}\leq \dfrac{d_0^2}{\lambda_1 \theta^\frac{2}{p-1}(1-\sigma^2)}=:c.
\end{align*}
Now using Lemma \ref{lm:carinhoso} with $c>0$ as above, $q=\frac{p+1}{p-1}$ and 
$\lambda_j\leftarrow 2\mu\lambda_j$, we find
\begin{align*}
 \prod_{j=2}^{k+1}\left(1+2\mu\lambda_{j}\right)&\geq \left(1+\left(\dfrac{(2\mu)^{\frac{p+1}{p-1}}}{c}k
  \right)^{\frac{p-1}{p+1}}\right)^{k}\\[2mm]
     & =\left(1+\dfrac{2\mu}{c^{\frac{p-1}{p+1}}}k^{\frac{p-1}{p+1}}\right)^{k}\\[2mm]
     & = \left(1+\left(\dfrac{2\mu\lambda_1^{\frac{p-1}{p+1}}\theta^{\frac{2}{p+1}}(1-\sigma^2)^{\frac{p-1}{p+1}}}{d_0^{\frac{2(p-1)}{p+1}}}\right)k^{\frac{p-1}{p+1}}\right)^{k},
\end{align*}
which, in turn, combined with \eqref{eq:sting2} and the definition of $C$ in \eqref{eq:furry2} finishes the proof of the lemma.
\end{proof}

\mgap

Next is the main result on global convergence rates for Algorithm \ref{alg:ls}. As we mentioned before, it provides a global \emph{superlinear} $\mathcal{O}\left(k^{\,-k\left(\frac{p-1}{p+1}\right)}\right)$ convergence, where $p-1\geq 1$ is the power in the high-order large-step condition (third inequality in \eqref{eq:alg_err2}).

\mgap

\begin{theorem}[{\bf Convergence rates for Algorithm \ref{alg:ls}}]
 \label{th:main_comp02}
 Consider the sequences evolved by \emph{Algorithm \ref{alg:ls}}, let $x^*$ denote the (unique) solution of
 \eqref{eq:co} and let $C>0$ be as in \eqref{eq:furry2}.
 Then the following holds:
 \begin{itemize}
 %
\item[\emph{(a)}] For all $k\geq 0$,
 \begin{align*}
   &h(y^{k+1})-h(x^*)\leq \dfrac{d_0^2}{2\lambda_1\left(1+\frac{2\mu C}{d_0^{\frac{2(p-1)}{p+1}}}\,k^{\,\left(\frac{p-1}{p+1}\right)}\right)^{k}}
   =\mathcal{O}\left(\frac{1}{k^{k\left(\frac{p-1}{p+1}\right)}}\right),\\[2mm]
   &\max\left\{\norm{x^*-x^{k+1}}^2,\norm{x^*-y^{k+1}}^2\right\}\leq 
   \dfrac{d_0^2}
   {\mu \lambda_1\left(1+ \frac{2\mu C}{d_0^{\frac{2(p-1)}{p+1}}} k^{\,\left(\frac{p-1}{p+1}\right)}\right)^{k}}
   =\mathcal{O}\left(\frac{1}{k^{k\left(\frac{p-1}{p+1}\right)}}\right).
 \end{align*}
\item[\emph{(b)}] For all $k\geq 1$, 
\begin{align*}
\begin{cases}
&v^{k+1}\in \partial_{\varepsilon_{k+1}} f(y^{k+1})+\partial g(y^{k+1}),\\[3mm]
&\norm{v^{k+1}}^2\leq 
\left(1+\sigma\sqrt{1+\mu C d_0^{-\frac{2(p-1)}{p+1}}}\right)^2
\dfrac{6 d_0^{\frac{2(3p-1)}{p+1}}}
   {\mu C^2\lambda_1\left(1+ \frac{2\mu C}{d_0^{\frac{2(p-1)}{p+1}}} (k-1)^{\,\left(\frac{p-1}{p+1}\right)}\right)^{k-1}}
   =\mathcal{O}\left(\frac{1}{(k-1)^{(k-1)\left(\frac{p-1}{p+1}\right)}}\right),\\[6mm]
&\varepsilon_{k+1}\leq  
\dfrac{3\sigma^2 d_0^{\frac{4p}{p+1}}}
   {\mu C \lambda_1\left(1+ \frac{2\mu C}{d_0^{\frac{2(p-1)}{p+1}}} (k-1)^{\,\left(\frac{p-1}{p+1}\right)}\right)^{k-1}}
   =\mathcal{O}\left(\frac{1}{(k-1)^{(k-1)\left(\frac{p-1}{p+1}\right)}}\right).
\end{cases}
\end{align*}
 \end{itemize}
\end{theorem}
\begin{proof}
Both items follow from Theorem \ref{th:main_comp} and Lemmas \ref{lm:emicida} and \ref{lm:belchior}. To prove the inequalities in item (b), one also has to use
the fact that the scalar function $t\mapsto \frac{1+\sigma\sqrt{1+\mu t}}{t}$ is nonincreasing as well as the lower bound on
$\lambda_k$ given in \eqref{eq:furry2}.
\end{proof}


\section{Applications to accelerated high-order tensor methods for strongly convex objectives}
 \label{sec:second}
In this section, we consider the problem
\begin{align}
 \label{eq:co4}
  \min_{x\in \HH}\,\left\{h(x):=f(x)+g(x)\right\},
\end{align}
 where $f, g:\HH\to (-\infty,\infty]$ are proper, closed and convex functions, $\mbox{dom}\,h\neq \emptyset$, and
 $g$ is {\it $\mu$-strongly convex} on $\HH$ and $p\geq 2$ {\it times continuously differentiable}  
 on $\Omega\supseteq \mbox{Dom}\,(\partial f)$ 
with $D^p g(\cdot)$ being {\it $L_p$-Lipschitz continuous on $\Omega$}: $0<L_p<+\infty$ and
 \begin{align}
  \label{eq:mg}
    \norm{D^p g(x) - D^p g(y)}\leq L_p \norm{x-y},\qquad  \forall x,y\in \Omega.
 \end{align}

 Define
 \begin{align}
  g_{x,p}(y):=g(x)+\sum_{k=1}^p\,\dfrac{1}{k!} D^k g(x)[y-x]^k + \dfrac{M}{(p+1)!}\norm{y-x}^{p+1}
  ,\qquad (x,y)\in \Omega\times \HH,
 \end{align}
 where $M>0$ is such that $M\geq p L_p$.
 
 As observed by Nesterov in \cite{nes-imp.mp20}, the function $g_{x,p}(\cdot)$ is convex whenever $M\geq p L_p$ and, moreover,
  \begin{align}
  \label{eq:error_p}
  \norm{\nabla g(y) -  \nabla g_{x,p}(y)}\leq \dfrac{L_p+M}{p!}\norm{y-x}^p,\qquad \forall (x,y)\in \Omega\times \HH.
 \end{align}
 
 At each iteration of the (exact) Proximal-Tensor method for solving \eqref{eq:co4} one has to find
 $y\in \HH$ solving an inclusion of the form
 \begin{align}
  \label{eq:manu}
   0\in \lambda\Big(\partial f(y)+ \nabla g_{z,p}(y)\Big) + y - x,
 \end{align} 
 where $z=P_{\Omega}(x)$ and $\lambda>0$. Note also that \eqref{eq:manu} is equivalent to solving the convex problem
 \begin{align}
  \label{eq:manux}
  \min_{y\in \HH}\,\left\{f(y)+ g_{z,p}(y)+\dfrac{1}{2\lambda}\norm{y-x}^2\right\}.
 \end{align}

Next definition introduces a notion of \emph{relative-error} inexact solution for \eqref{eq:manu} (or, equivalently, \eqref{eq:manux}). It will be used in step 2 (see \eqref{eq:sigma11}) of Algorithm \ref{alg:second}, and can be motivated as follows.
Note that the inclusion \eqref{eq:manu} can be splitted as an inclusion-equation system:
\begin{align*}
\begin{cases}
 u\in \partial f(y),&\\[2mm]
 \lambda\Big(u+\nabla g_{z,p}(y)\Big) + y - x = 0.
\end{cases}
\end{align*}
Definition \ref{def:newton_sol} below relaxes the latter conditions by allowing errors in both the inclusion 
($u$ will belong to $\partial_\varepsilon f(y)$ instead of belonging to $\partial f(y)$) and the equation.

\begin{definition}
 \label{def:newton_sol}
  The triple $(y,u,\varepsilon)\in \HH\times \HH\times \R_+$ is a $\hat \sigma$-approximate \emph{Tensor solution} of \eqref{eq:manu} at $(x,\lambda)\in \HH\times \R_{++}$ if
  $\hat \sigma\geq 0$ and
  \begin{align}
   \label{eq:mbarros2}
   u\in \partial_\varepsilon f(y),\qquad 
   \dfrac{\left\|\lambda \Big(u+ \nabla g_{z,p}(y)\Big)+y-x\right\|^2}{1+\lambda\mu}+2\lambda\varepsilon
   \leq \hat\sigma^2\norm{y-x}^2,
  \end{align}
 where $z=P_{\Omega}(x)$.
\end{definition}

Note that if $\hat\sigma=0$ in \eqref{eq:mbarros2}, then it follows that
$\varepsilon=0$, $u\in \partial f(y)$ and $\lambda \Big(u+ \nabla g_{z,p}(y)\Big)+y-x=0$, which implies that
$y$ is the solution of \eqref{eq:manu}.
We also mention that if we set $\mu=0$ in Definition \ref{def:newton_sol} then we recover 
\cite[Definition 2.1]{zhang-preprint20} (see also \cite[Definition 1]{jordan-controlpreprint20}).

\mgap

Next proposition shows that $\hat\sigma$-approximate solutions of \eqref{eq:manu} provide relative-error approximate solutions in the sense of \eqref{eq:alg_err2}.

\begin{proposition}
 \label{pr:tensor_hat}
Let $(u,y,\varepsilon)$ be a $\hat\sigma$-approximate Tensor solution of \eqref{eq:manu} at
$(x,\lambda)\in \HH\times \R_{++}$ \emph{(}in the sense of \emph{Definition \ref{def:newton_sol}}\emph{)} and define
\begin{align}
 \label{eq:gulliver}
 v=u + \nabla g(y),\qquad 
 \sigma=\dfrac{\lambda (L_p+M)\norm{y-x}^{p-1}}{p!\sqrt{1+\lambda\mu}}+\hat \sigma.
\end{align}
Then,
\begin{align}
 \label{eq:gulliver02}
 v\in \partial_\varepsilon f(y)+\nabla g(y),
 \qquad \dfrac{\norm{\lambda v+y-x}^2}{1+\lambda\mu}+2\lambda\varepsilon\leq \sigma^2\norm{y-x}^2.
\end{align}
\end{proposition}
\begin{proof}
Note that the inclusion in \eqref{eq:gulliver02} follows from the definition of $v$ in \eqref{eq:gulliver} and the inclusion in \eqref{eq:mbarros2}.
To prove the inequality in \eqref{eq:gulliver02}, note that from the definition of $v$ in \eqref{eq:gulliver}, the triangle inequality and property \eqref{eq:error_p}, we find
\begin{align*}
%
\nonumber
\norm{\lambda v+y-x}^2&=\norm{\lambda \left(u+\nabla g_{z,p}(y)\right)+y-x+\lambda \big(\nabla g(y) - \nabla g_{z,p}(y)\big)}^2\\
\nonumber
 & \leq \Big(\norm{\lambda \left(u+\nabla g_{z,p}(y)\right)+y-x}+\lambda\norm{\nabla g(y) - \nabla g_{z,p}(y)}\Big)^2\\
 \nonumber
 &\leq \left(\norm{\lambda \left(u+\nabla g_{z,p}(y)\right)+y-x}+ \frac{\lambda (L_p+M)}{p!} \norm{y-z}^p\right)^2\\
 &\leq \left(\norm{\lambda \left(u+\nabla g_{z,p}(y)\right)+y-x}+ \frac{\lambda (L_p+M)}{p!} \norm{y-x}^p\right)^2,
\end{align*}
where in the last inequality we also used the fact that $\norm{y-z}\leq \norm{y-x}$.
(because $y\in \mbox{Dom}(\partial_\varepsilon f)\subset \overline{\mbox{Dom}(\partial f)}\subset \Omega$ and 
$z=P_\Omega(x)$.)

Hence,
\begin{align*}
\dfrac{\norm{\lambda v+y-x}^2}{1+\lambda\mu}+2\lambda\varepsilon
    &\leq \left(\dfrac{\norm{\lambda \left(u+\nabla g_{z,p}(y)\right)+y-x}}{\sqrt{1+\lambda\mu}}+ 
     \frac{\lambda (L_p+M)}{p!\sqrt{1+\lambda\mu}} \norm{y-x}^p\right)^2+2\lambda\varepsilon.
\end{align*}
Using now the elementary inequality $(a+b)^2+c\leq \left(b+\sqrt{a^2+c}\right)^2$ 
with 
$a=\frac{\norm{\lambda \left(u+\nabla g_{z,p}(y)\right) +y-x}}{\sqrt{1+\lambda\mu}}$,
$b=\lambda (L_p+M)\norm{y-x}^p / (p!\sqrt{1+\lambda\mu})$ 
 and $c=2\lambda\varepsilon$, we find
\begin{align*}
\dfrac{\norm{\lambda v+y-x}^2}{1+\lambda\mu}+2\lambda\varepsilon&\leq
\left(\frac{\lambda (L_p+M)}{p!\sqrt{1+\lambda\mu}} \norm{y-x}^p+
\sqrt{\dfrac{\norm{\lambda \left(u+\nabla g_{z,p}(y)\right)+y-x}^2}{1+\lambda\mu}+2\lambda\varepsilon}\right)^2\\
&\leq \left(\frac{\lambda (L_p+M)}{p!\sqrt{1+\lambda\mu}} \norm{y-x}^p+\hat \sigma\norm{y-x}\right)^2\\
&=\left(\frac{\lambda (L_p+M)}{p!\sqrt{1+\lambda\mu}} \norm{y-x}^{p-1}+\hat \sigma\right)^2\norm{y-x}^2\\
&=\sigma^2\norm{y-x}^2,
\end{align*}
where in the second inequality we used the inequality in \eqref{eq:mbarros2} and in the second identity we used the second
equality \eqref{eq:gulliver}.
\end{proof}

\mgap

Next we present our $p$-th order inexact (relative-error) accelerated tensor algorithm for solving \eqref{eq:co4}.

\mgap
\mgap

%
\noindent
\fbox{
\begin{minipage}[h]{6.6 in}
\begin{algorithm}
\label{alg:second}
{\bf An accelerated inexact high-order tensor method for solving \eqref{eq:co4}}
\end{algorithm}
\begin{itemize}
\item [0)] 
Choose $x^0,y^0\in \HH$ and $p\geq 2$, $\hat \sigma\geq 0$, $0<\sigma_\ell<\sigma_u<1$ such that
\begin{align}
 \label{eq:sigma10}
  \sigma:=\sigma_u+\hat\sigma<1,\qquad \sigma_\ell(1+\hat\sigma)^{p-1}<\sigma_u(1-\hat\sigma)^{p-1};
\end{align}
let $A_0=0$ and set $k=0$.
%
%
\item [1)] Compute $\lambda_{k+1}>0$ and a $\hat\sigma$-approximate Tensor solution $(u^{k+1},y^{k+1},\varepsilon_{k+1})$ (in the sense of Definition \ref{def:newton_sol}) of \eqref{eq:manu} at $(\widetilde x^k,\lambda_{k+1})$ satisfying
\begin{align}
 \label{eq:sigma11}
\dfrac{p!\,\sigma_\ell}{L_p+M}\leq  
 \lambda_{k+1}\norm{y^{k+1}-\widetilde x^k}^{p-1}\leq 
\dfrac{p!\,\sigma_u\sqrt{1+\lambda_{k+1}\mu}}{L_p+M},
\end{align}
where
\begin{align}
\label{eq:alg_xtil5}
& \widetilde x^k =  \left(\dfrac{a_{k+1}-\mu A_k\lambda_{k+1}}{A_k+a_{k+1}}\right)x^k
 + \left(\dfrac{A_k+\mu A_k\lambda_{k+1}}{A_k+a_{k+1}}\right)y^k,\\[3mm]
 \label{eq:alg_a5}
& a_{k+1}=\dfrac{(1+2\mu A_k)\lambda_{k+1}+\sqrt{(1+2\mu A_k)^2\lambda_{k+1}^2
+4(1+\mu A_k)A_k\lambda_{k+1}}}{2}. 
\end{align}

\item[2)] 
 Let
  \begin{align}
  \label{eq:alg_A5}
  &A_{k+1} = A_k + a_{k+1},\\[2mm]
  \label{eq:alg_nv}
  &v^{k+1}=u^{k+1}-\nabla g_{z^k,p}(y^{k+1})+\nabla g(y^{k+1}),\quad z^k=P_{\Omega}(\widetilde x^k),\\[2mm]
  \label{eq:alg_nv2}
  & x^{k+1} = \left(\dfrac{1+\mu A_k}{1+\mu A_{k+1}}\right)x^k + 
  \left(\dfrac{\mu a_{k+1}}{1+\mu A_{k+1}}\right) y^{k+1}-\left(\dfrac{a_{k+1}}{1+\mu A_{k+1}}\right)v^{k+1}.
  \end{align}
\item[3)] 
Set $k=k+1$ and go to step 1.
\end{itemize}
\noindent
\end{minipage}
} 
%
%
\mgap
\mgap

We now make two remarks concerning Algorithm \ref{alg:second}:

\begin{itemize}
\item[(i)] Algorithm \ref{alg:second} is a generalization of \cite[Algorithm 3]{jordan-controlpreprint20} for strongly convex problems. 
The latter algorithm, which can be applied to \eqref{eq:co4} with
$f\equiv 0$, has the global convergence rates
$\mathcal{O}\left(k^{-\frac{3p+1}{2}}\right)$ and 
$\mathcal{O}\left(k^{-3p}\right)$ for (in the notation of this paper) $g(y^k)-g(x^*)$
and $\inf_{1\leq i\leq k}\,\norm{\nabla g(y^i)}^2$, respectively (see \cite[Theorem 4.13]{jordan-controlpreprint20}).

In contrast to this, here we obtained, see Theorem \ref{th:main_tensor}, the fast global 
$\mathcal{O}\left(k^{\,-k\left(\frac{p-1}{p+1}\right)}\right)$ convergence rate. 
\item[(ii)] We also mention that a $\hat\sigma$-approximate Tensor solution satisfying \eqref{eq:sigma11} can be computed using bisection schemes (see \cite{alv.mon.sva-sqp.pre14,zhang-preprint20,mon.sva-acc.siam13}).
More precisely, by defining the curve
\begin{align*}
 \tau(\lambda) = \frac{a_{k+1}(\lambda)-\mu A_k \lambda}{A_k + a_{k+1}(\lambda)},\quad \lambda>0,
\end{align*}
where
\begin{align*}
 a_{k+1}(\lambda) := \dfrac{(1+2\mu A_k)\lambda + \sqrt{(1+2\mu A_k)^2\lambda^2
+4(1+\mu A_k)A_k\lambda}}{2},
\end{align*}
we find that $\widetilde x^k$ as in (67) can be written as
\begin{align*}
\widetilde x^k = y^k + \tau_{k+1}(x^k-y^k),
\end{align*}
where $\tau_{k+1}:=\tau(\lambda_{k+1})$.
Moreover, it is not difficult to check that  $\tau(\cdot)$ is (smooth) increasing and that $\dot{\tau}(\lambda)$ satisfies
\[
 \dot{\tau}(\lambda)\leq \frac{1}{\lambda}\qquad \forall \lambda>0.
\]
As a consequence of the above inequality, one can reproduce Lemma 7.13 as in~\cite[p. 1121]{mon.sva-acc.siam13} and pose the ``Line search problem'' as in~\cite[p. 1117]{mon.sva-acc.siam13} (with $\lambda\norm{y_\lambda - x(\lambda)}$ replaced by $\lambda\norm{y_\lambda - x(\lambda)}^{p-1}$).
The general search procedures studied in references~\cite[Section 4]{alv.mon.sva-sqp.pre14} and~\cite{zhang-preprint20} would also be helpful. The complexity of the 
bracketing/bisection procedure depends on a logarithm of the inverse of the precision (see, e.g.,~\cite[Theorem 7.16]{mon.sva-acc.siam13}).

\end{itemize}

\mgap

\begin{proposition}
 \label{pr:tensor_special}
  \emph{Algorithm \ref{alg:second}} is a special instance of \emph{Algorithm \ref{alg:ls}} for solving \eqref{eq:co4}, where
  \begin{align}
   \label{eq:theta_tensor}
    \theta:=\dfrac{p!\sigma_\ell}{L_p+M}.
  \end{align}
\end{proposition}
\begin{proof}
It follows from the definitions of Algorithms \ref{alg:ls} and \ref{alg:second} that we only have to prove that \eqref{eq:alg_err2} holds.
Note that the inclusion and the first inequality in \eqref{eq:alg_err2} follow from step 1 of Algorithm \ref{alg:second} 
-- the fact that $(u^{k+1},y^{k+1},\varepsilon_{k+1})$ is a $\hat\sigma$-approximate Tensor solution of \eqref{eq:manu}--, the second
inequality in \eqref{eq:sigma11}, the definition of $\sigma$ in \eqref{eq:sigma10} and
Proposition \ref{pr:tensor_hat}.
To finish the proof of the proposition, note that the last inequality in \eqref{eq:alg_err2} (the large-step condition) is a direct consequence of the first inequality in \eqref{eq:sigma11} and \eqref{eq:theta_tensor}.
\end{proof}

\mgap

Next theorem states the fast global $\mathcal{O}\left(k^{\,-k\left(\frac{p-1}{p+1}\right)}\right)$ convergence rate for Algorithm \ref{alg:second}.

\mgap

\begin{theorem}[{\bf Convergence rates for Algorithm \ref{alg:second}}]
 \label{th:main_tensor}
 Consider the sequences generated by \emph{Algorithm \ref{alg:second}}, let $\theta>0$ be as in \eqref{eq:theta_tensor} and let $C>0$ be as in \eqref{eq:furry2}, where $d_0:=\norm{x^0-x^*}$ and $x^*$ is the (unique) solution of
 \eqref{eq:co4}.
 
 Then all the conclusions of \emph{Theorem \ref{th:main_comp02}} hold.
\end{theorem}
\begin{proof}
The proof follows from Proposition \ref{pr:tensor_special} and Theorem \ref{th:main_comp02}.
\end{proof}

\mgap

\noindent
{\bf Remarks.} We now make some remarks concerning Algorithm \ref{alg:second}:
\begin{itemize}
\item[(i)] Note that if $p=2$ in \eqref{eq:mg}, then it follows that \emph{Algorithm \ref{alg:second}} reduces to 
an (accelerated) inexact proximal-Newton-type algorithm with $\mathcal{O}\left(k^{-k/3}\right)$ global convergence rate 
(see Theorems \ref{th:main_tensor} and \ref{th:main_comp02}).
\item[(ii)] In \cite{gra.nes-acc.siam19}, an accelerated regularized Newton method was proposed for solving \eqref{eq:co4} 
with $p=2$ and $g$ being $\sigma_q$-uniformly convex. In the notation of this paper, Algorithm 1 of \cite{gra.nes-acc.siam19} (see Theorem 3.3) has global linear
\[
 \mathcal{O}\left(\dfrac{L_2 d_0^3}{\left(1+\left(\frac{\mu}{L_2}\right)^{1/3}\right)^{2(k-1)}}\right)
\]
convergence rate for function values.
\item[(iii)] In \cite{gasnikov-optimal.pmlr19}, problem \eqref{eq:co4} was considered with $f\equiv 0$ and $g$ assumed to be 
$\sigma_q$-uniformly convex. The complexity of an optimal tensor method with restart 
(\cite[Algorithm 3]{gasnikov-optimal.pmlr19}) to compute $x$ satisfying $g(x)-g(x^*)\leq \varepsilon$ is 
\begin{align*}
&\mathcal{O}\left(\left(\dfrac{L_p}{\sigma_{p+1}}\right)^{\frac{2}{3p+1}}\log_2\left(\frac{\Delta_0}{\varepsilon}\right)\right), 
\;q = p+1; \quad\\[2mm] 
&\mathcal{O}\left(\left(\dfrac{L_p\left(\Delta_0\right)^{p+1-q}{q}}{\sigma_q^{\frac{p+1}{q}}}\right)^{\frac{2}{3p+1}}
+ \log_2\left(\dfrac{\Delta_0}{\varepsilon}\right)
\right),\; q<p+1,
\end{align*}
where $\Delta_0\geq g(x^0)-g(x^*)$. We mention that the upper-bound $\Delta_0$ on $g(x^0)-g(x^*)$ is assumed to be known in the implementation of \cite[Algorithm 3]{gasnikov-optimal.pmlr19}.
\item[(iv)] In \cite{dvurechensky-near.preprint19}, a near-optimal algorithm for solving \eqref{eq:co4} with $f\equiv 0$ and $\mu=0$ (i.e., with $g$ convex but not strongly convex) was proposed and studied. The iteration-complexity of Algorithm 2 in
\cite{dvurechensky-near.preprint19} to find $x$ satisfying 
$\norm{\nabla g(x)}\leq \varepsilon$ (see \cite[Theorem 2]{dvurechensky-near.preprint19}) is
\[
 \mathcal{O}\left(\dfrac{L_p^{\frac{2}{3p+1}}}{\varepsilon^{\frac{2(p+1)}{3p+1}}}
 \Delta_0^{\frac{2p}{3p+1}}
 +\log_2\left(\dfrac{2^{\frac{4p-3}{p+1}}\Delta_0\left(pL_p\right)^{\frac{1}{p}}(p+1)!}
 {\varepsilon^{\frac{p}{p+1}}}\right),
 \right)
\]
where $\Delta_0\geq g(x^0)-g(x^*)$. Analogously to \cite{gasnikov-optimal.pmlr19}, the upper-bound $\Delta_0$ is assumed to be known while running \cite[Algorithm 2]{dvurechensky-near.preprint19}.
\item[(v)]
In \cite{gra.nes-ten.oms20}, tensor methods for solving \eqref{eq:co4} with $f\equiv 0$ and $g$ being $p$-times continuously differentiable with $\nu$-H\"older continuous $p$th derivatives were proposed. The iteration-complexity is 
$\mathcal{O}\left(\varepsilon^{-1/(p+\nu-1)}\right)$ and 
$\mathcal{O}\left(\varepsilon^{-(p+\nu)/[(p+\nu-1)(p+\nu+1)]}\right)$ for the non accelerated and accelerated methods, respectively, to find $x$ such that $\norm{\nabla g(x)}\leq \varepsilon$.
\item[(vi)]
In \cite{doi.nes-min.jota21}, a regularized Newton method with linear convergence rate in function values for 
minimizing uniformly convex functions with $\nu$-H\"older continuous Hessian was proposed and studied. 
The main result (\cite[Theorem 4.1]{doi.nes-min.jota21}) gives that the rate is
linear:
\[
 F(x_{k+1})-F^*\leq \left(1-\min\left\{
 \dfrac{(2+\nu)\left((1+\nu)(2+\nu)\right)^{1/(1+\nu)}\left(\gamma_f(\nu)\right)^\frac{1}{1+\nu}}
 {(1+\nu)6^{3/2}\cdot 2^{1/2}\cdot (8+\nu)^{(1-\nu)/(2+2\nu)}}
 ,\frac{1}{2}
 \right\}\right)\left(F(x_k)-F^*\right),
\] 
where $F$ is the objective function and $F^*$ denotes the optimal value.
\end{itemize}

\section{Applications to first-order methods for strongly convex problems}
 \label{sec:first}
Consider the convex optimization problem
\begin{align}
 \label{eq:co2}
  \min_{x\in \HH}\,\left\{h(x):=f(x)+g(x)\right\},
\end{align}
 where $f, g:\HH\to (-\infty,\infty]$ are proper, closed and convex functions, $\mbox{dom}\,h\neq \emptyset$, and, additionally,
 $g$ is \emph{$\mu$-strongly convex} on $\HH$ and {\it differentiable} on $\Omega\supseteq \mbox{dom}\,f$ 
with $\nabla g$ being $L$\emph{-Lipschitz continuous on $\Omega$}

 An iteration of the proximal-gradient (forward-backward) method for solving \eqref{eq:co2} can be written as follows:
\begin{align}
 \label{eq:ite.fb}
 y = (\lambda \partial f+I)^{-1} \left(x-\lambda \nabla g(z)\right),
\end{align}
where $z=P_{\Omega}(x)$ and $\lambda>0$. Using the definition of $(\lambda \partial f+I)^{-1}$, it is easy to see that
\eqref{eq:ite.fb} is equivalent to solving the inclusion
\begin{align}
 \label{eq:vhalen}
  0\in \lambda \Big(\partial f(y)+\nabla g(z)\Big) + y - x.
\end{align}

Next we define a notion of approximate solution for \eqref{eq:vhalen} within a \emph{relative-error} criterion.

\begin{definition}
 \label{def:gradient_sol}
  The triple $(y,u,\varepsilon)\in \HH\times \HH\times \R_+$ is a $\hat \sigma$-approximate 
  Proximal-Gradient (PG) solution of \eqref{eq:vhalen} at $(x,\lambda)\in \HH\times \R_{++}$ if
  $\hat \sigma\geq 0$ and
  \begin{align}
   \label{eq:mbarros4}
   u\in \partial_\varepsilon f(y),\qquad 
   \dfrac{\norm{\lambda \left(u+\nabla g(z)\right)+y-x}^2}{1+\lambda\mu}+2\lambda\varepsilon
   \leq \hat\sigma^2\norm{y-x}^2,
  \end{align}
 where $z=P_{\Omega}(x)$.
 We also write 
 \begin{align*}
   (y,u,\varepsilon) \approx (\lambda \partial f+I)^{-1} \left(x-\lambda \nabla g(z)\right)
 \end{align*}
 to mean that $(y,u,\varepsilon)$ is a $\hat \sigma$-approximate PG solution of
 \eqref{eq:vhalen} at $(x,\lambda)$.
\end{definition}

Note that if $\hat\sigma=0$ in \eqref{eq:mbarros4}, then it follows that $\varepsilon=0$, $u\in \partial f(y)$ and
$\lambda\left[u+\nabla g(z)\right]+y-x=0$, which implies that $y$ is the (exact) solution of \eqref{eq:vhalen}. In particular, in this case, $y$ satisfies \eqref{eq:ite.fb}.

\mgap

\begin{proposition}
 \label{pr:vhalen}
Let $(u,y,\varepsilon)$ be a $\hat\sigma$-approximate PG solution of \eqref{eq:vhalen} at 
$(x,\lambda)\in \HH\times \R_{++}$ as in \emph{Definition \ref{def:gradient_sol}} and define
\begin{align}
 \label{eq:vhalen2}
 v=u +\nabla g(y),\qquad \sigma=\dfrac{\lambda L}{\sqrt{1+\lambda\mu}}+\hat \sigma.
\end{align}
Then,
\begin{align}
 \label{eq:vhalen3}
 v\in \partial_\varepsilon f(y)+\nabla g(y),
 \qquad \dfrac{\norm{\lambda v+y-x}^2}{1+\lambda\mu}+2\lambda\varepsilon\leq \sigma^2\norm{y-x}^2.
\end{align}
\end{proposition}
\begin{proof}
The proof follows the same outline of Proposition \ref{pr:tensor_hat}'s proof.
\end{proof}

\mgap

 For solving \eqref{eq:co2}, we propose the following inexact (relative-error) accelerated first-order algorithm.

\mgap
\mgap

%
\noindent
\fbox{
\begin{minipage}[h]{6.6 in}
\begin{algorithm}
\label{alg:first01}
{\bf An accelerated inexact proximal-gradient algorithm for solving \eqref{eq:co2}}
\end{algorithm}
\begin{itemize}
\item [0)] 
Choose $x^0,y^0\in \HH$ and $\hat\sigma\geq 0$,  $0<\sigma_u\leq 1$ such
that $\sigma:=\sigma_u+\hat\sigma<1$ and let
\begin{align}
 \label{eq:first_lambda}
\lambda=\dfrac{\sigma_u}{\sqrt{\left(\dfrac{\sigma_u\mu}{2}\right)^2+L^2}-\dfrac{\sigma_u\mu}{2}}>\dfrac{\sigma_u}{L};
 \end{align}
 let $A_0=0$ and set $k=0$.
\item [1)] Compute $z^k=P_{\Omega}(\widetilde x^k)$ and 
\begin{align}
\begin{aligned}
 \label{eq:alg_err4}
%
  \qquad (y^{k+1},u^{k+1},\varepsilon_{k+1})\approx (\lambda \partial f+I)^{-1}\left(\widetilde x^k-\lambda \nabla g(z^k)\right),
\end{aligned}
\end{align}
i.e., compute a $\hat\sigma$-approximate PG solution $(u^{k+1},y^{k+1},\varepsilon_{k+1})$ 
at $(\widetilde x^k,\lambda)$ (in the sense of Definition \ref{def:gradient_sol}),
where
\begin{align}
\label{eq:alg_xtil4}
& \widetilde x^k =  \left(\dfrac{a_{k+1}-\mu A_k\lambda}{A_k+a_{k+1}}\right)x^k
 + \left(\dfrac{A_k+\mu A_k\lambda}{A_k+a_{k+1}}\right)y^k,\\[3mm]
 \label{eq:alg_a4}
& a_{k+1}=\dfrac{(1+2\mu A_k)\lambda+\sqrt{(1+2\mu A_k)^2\lambda^2
+4(1+\mu A_k)A_k\lambda}}{2}. 
\end{align}

\item[2)] 
 Let
  \begin{align}
  \label{eq:alg_A4}
  &A_{k+1} = A_k + a_{k+1},\\[2mm]
  \label{eq:alg_vn4}
  &v^{k+1}=u^{k+1}+\nabla g(y^{k+1}),\\[2mm]
  \label{eq:alg_xne4}
  & x^{k+1} = \left(\dfrac{1+\mu A_k}{1+\mu A_{k+1}}\right)x^k + 
  \left(\dfrac{\mu a_{k+1}}{1+\mu A_{k+1}}\right) y^{k+1}- \left(\dfrac{a_{k+1}}{1+\mu A_{k+1}}\right) v^{k+1}.
  \end{align}
\item[3)] 
Set $k=k+1$ and go to step 1.
\end{itemize}
\noindent
\end{minipage}
} 
%

\mgap
\mgap

\noindent
We now make the following remark concerning Algorithm \ref{alg:first01}:
\begin{itemize}
\item[(i)] From the definition of $\lambda>0$ in \eqref{eq:first_lambda} we obtain
\begin{align}
 \label{eq:watson}
 \dfrac{\lambda^2 L^2}{1+\lambda\mu}=\sigma_u^2.
\end{align}
Indeed, it is easy to check that $\lambda>0$ is the largest root of 
$L^2\lambda^2-(\sigma_u^2\mu)\lambda-\sigma_u^2=0$, which is clearly equivalent to \eqref{eq:watson}.
Now using \eqref{eq:watson}, we find
\begin{align}
 \label{eq:watson2}
 \sigma=\sigma_u+\hat\sigma=\dfrac{\lambda L}{\sqrt{1+\lambda \mu}}+\hat\sigma.
\end{align}
\end{itemize}

\mgap

Next proposition shows that Algorithm \ref{alg:first01} is a special instance of Algorithm \ref{alg:main} for solving
\eqref{eq:co2}.

\mgap

\begin{proposition}
 \label{pr:special01}
 Consider the sequences evolved by \emph{Algorithm \ref{alg:first01}} and let $\lambda_{k+1}\equiv \lambda$.
Then, $\lambda_{k+1}>0$ and the triple $(y^{k+1}, v^{k+1}, \varepsilon_{k+1})$ satisfy condition \eqref{eq:alg_err}
in \emph{Algorithm \ref{alg:main}} with $\sigma=\hat\sigma+\sigma_u$. As a consequence, \emph{Algorithm \ref{alg:first01}} is a special instance
of \emph{Algorithm \ref{alg:main}} for solving \eqref{eq:co2}.
\end{proposition}
\begin{proof}
The proof follows from \eqref{eq:alg_err4}, \eqref{eq:watson2}, Proposition \ref{pr:vhalen} and the definitions of Algorithms 
\ref{alg:main} and \ref{alg:first01}.
\end{proof}

\mgap

Next we summarize the results on \emph{linear convergence} rates for Algorithm \ref{alg:first01}.

\begin{theorem}[{\bf Convergence rates for Algorithm \ref{alg:first01}}]
 \label{th:first01}
 Consider the sequences evolved by \emph{Algorithm \ref{alg:first01}} and let $\sigma=\hat\sigma+\sigma_u$.
  Let also $x^*$ be the unique solution of \eqref{eq:co2}, let $d_0$ be as in  \eqref{eq:def.d0} and denote
  $\gamma=\sqrt{(1+\sigma_u)^{-1}\sigma_u}$.
 The following holds:
 \begin{itemize}
 \item[\emph{(a)}] For all $k\geq 1$,
 \begin{align*}
  & h(y^{k})-h(x^*)\leq \dfrac{L d_0^2}{2\sigma_u}\left(1-\gamma\sqrt{\dfrac{\mu}{L}}\,\right)^{k-1},\\[3mm]
  & \max\left\{\norm{x^*-y^k},\norm{x^*-x^k}\right\}\leq 
      \sqrt{\dfrac{L}{\sigma_u\mu}}\,d_0\left(1-\gamma\sqrt{\dfrac{\mu}{L}}\,\right)^{(k-1)/2}.
 \end{align*}
\item[\emph{(b)}] For all $k\geq 1$, 
\begin{align*}
%
\begin{cases}
&v^{k+1}\in \partial_{\varepsilon_{k+1}} f(y^{k+1})+\partial g(y^{k+1}),\\[2mm]
&\norm{v^{k+1}}\leq 
\dfrac{6d_0 L^{3/2}}{\mu^{1/2}\sigma_u^{3/2}} 
\left(1+\sigma\sqrt{1+\frac{\sigma_u\,\mu}{L}}\,\right)
\left(1-\gamma\sqrt{\dfrac{\mu}{L}}\,\right)^{(k-1)/2},\\[6mm]
&\varepsilon_{k+1}\leq \dfrac{3\sigma^2 d_0^2 L^2}
{\sigma_u^2\,\mu}\left(1-\gamma\sqrt{\dfrac{\mu}{L}}\,\right)^{k-1}.
\end{cases}
\end{align*}

 \end{itemize}
\end{theorem}
\begin{proof}
(a) First note that simple computations using \eqref{eq:def.alpha} with $\underline{\lambda}=\lambda$, the inequality in \eqref{eq:first_lambda}, the definition of $\gamma>0$ and the
fact that $L \geq \mu$ show that 
\begin{align}
 \label{eq:viola}
 \alpha>\sqrt{(1+\sigma_u)^{-1}\sigma_u}\sqrt{\dfrac{\mu}{L}}=\gamma\sqrt{\dfrac{\mu}{L}},\qquad \lambda>\dfrac{\sigma_u}{L},
\end{align}
which combined with Proposition \ref{pr:special01} and Theorem \ref{th:main_comp03}(a) gives  the proof of (a). 

\mgap

(b) The result follows from \eqref{eq:viola}, Proposition \ref{pr:special01} and Theorem \ref{th:main_comp03}(b).
\end{proof}

\mgap
\mgap

%
\section*{Acknowledgments}
The author would like to thank Dr. Benar F. Svaiter for the fruitful discussions related to the first draft of this work.
The author also thank the three anonymous referees and the associate editor for their comments that significantly improved the manuscript.
%

\appendix

\section{Some auxiliary results}

\begin{lemma}
 \label{lm:carinhoso}
For all $k\geq 1$, the optimal value of the minimization problem, over $\lambda_1,\dots, \lambda_k>0$,
\begin{align}
 \label{eq:min_prod}
\begin{aligned}
 &\min \,\prod_{j=1}^{k}\left(1+\lambda_{j}\right)\\[2mm]
 &\emph{s.t.}\;\;\sum_{j=1}^k\dfrac{1}{\lambda_j^q}\leq c,
\end{aligned}
\end{align}
where $c>0$ and $q\geq 1$, is given by
\[
 \left(1+\left(\dfrac{k}{c}\right)^{1/q}\right)^k.
\]
\end{lemma}
\begin{proof}
First consider the convex problem
\begin{align}
 \label{eq:min_prod_log}
\begin{aligned}
 &\min \,\sum_{j=1}^{k}\,\log (1+e^{t_j})\\[2mm]
 &\mbox{s.t.}\;\;\sum_{j=1}^k\dfrac{1}{e^{qt_j}}\leq c.
\end{aligned}
\end{align}
Since the objective and constraint functions in \eqref{eq:min_prod_log} are convex and invariant under permutations on 
$(t_1,\dots, t_k)$, it follows that one of its solutions takes the form $(t,\dots, t)$. It is also clear that at any solution 
the inequality in \eqref{eq:min_prod_log} must hold as an equality. Hence, $k \frac{1}{e^{qt}}=c$, i.e.,
$e^t=\left(\frac{k}{c}\right)^{1/q}$. As a consequence, for all $(t_1,\dots, t_k)$ such that
$\sum_{j=1}^k\frac{1}{e^{qt_j}}\leq c$,
\begin{align}
 \label{eq:bourgain}
 \sum_{j=1}^{k}\,\log (1+e^{t_j})\geq k \log (1+e^t)=k \log \left(1+\left(\dfrac{k}{c}\right)^{1/q}\right)
 =\log\left(\left(1+\left(\dfrac{k}{c}\right)^{1/q}\right)^k\right).
\end{align}

Now let $\lambda_1,\dots, \lambda_k>0$ be such that $\sum_{j=1}^k\frac{1}{\lambda_j^q}\leq c$ and define
$t_j:=\log(\lambda_j)$, for $j\in \{1,\dots, k\}$. Then, since in this case $\sum_{j=1}^k\frac{1}{e^{qt_j}}\leq c$, using
\eqref{eq:bourgain} and some basic properties of logarithms we find
\begin{align*}
\prod_{j=1}^k\,(1+\lambda_j)=\prod_{j=1}^k\,(1+e^{t_j})&=e^{\log\left(\prod_{j=1}^k\,(1+e^{t_j})\right)}\\
   &=e^{\sum_{j=1}^{k}\,\log (1+e^{t_j})}\\
   &\geq e^{\log\left(\left(1+\left(\dfrac{k}{c}\right)^{1/q}\right)^k\right)}\\
   &=\left(1+\left(\dfrac{k}{c}\right)^{1/q}\right)^k,
\end{align*}
which concludes the proof of the lemma.
\end{proof}

\mgap
\mgap

\begin{lemma}
 \label{lm:wolfe}
The following holds for $q(\cdot)$ defined by 
\begin{align}
\label{eq:def.q}
q(x)=\inner{v}{x-y}+\frac{\mu}{2}\norm{x-y}^2-\varepsilon+\frac{1}{2\lambda}\norm{x-z}^2\qquad (x\in \HH)
\end{align}
where $v,y,z\in \HH$ and $\mu,\varepsilon,\lambda>0$.
\begin{itemize}
\item[\emph{(a)}] The (unique) global minimizer of $q(\cdot)$ is given by  
\[
 x^*=\frac{1}{1+\lambda\mu}z+\frac{\lambda\mu}{1+\lambda\mu}y
-\frac{\lambda}{1+\lambda\mu}v.
\]
\item[\emph{(b)}] We have,
\begin{align*}
%
\min_{x}\,q(x)=
\dfrac{1}{2\lambda}\left[\norm{y-z}^2
  -\left(\dfrac{\norm{\lambda v+y-z}^2}{1+\lambda\mu}+2\lambda\varepsilon\right)\right].
\end{align*}
\item[\emph{(c)}] We have,
\begin{align*}
q(x)=\dfrac{1}{2\lambda}\left[\norm{y-z}^2
  -\left(\dfrac{\norm{\lambda v+y-z}^2}{1+\lambda\mu}+2\lambda\varepsilon\right)\right]
+\dfrac{1+\lambda\mu}{2\lambda}\norm{x-x^*}^2,\qquad \forall x\in \HH.
\end{align*}
\end{itemize}
\end{lemma}
\begin{proof}
(a) This follows directly from \eqref{eq:def.q} and some simple calculus.

(b) 
Note first that 
\begin{align}
 \label{eq:opt.value}
 \min_x\,q(x)=q(x^*)
    &=\inner{v}{x^*-y}+\frac{\mu}{2}\norm{x^*-y}^2-\varepsilon+\frac{1}{2\lambda}\norm{x^*-z}^2.
 \end{align}
 Using the well-known  identity $a\norm{z}^2+b\norm{w}^2=\frac{1}{a+b}
 \left[\norm{az+bw}^2+ab\norm{z-w}^2\right]$ with $a=\mu$, $b=1/ \lambda$,
 $z=x^*-y$ and $w=x^*-z$, and (a) we find
 \begin{align}
  \label{eq:square.xstar}
  \nonumber
  \mu \norm{x^*-y}^2+\frac{1}{\lambda}\norm{x^*-z}^2&=
  \dfrac{\lambda}{1+\lambda\mu}\left[\left\|\underbrace{\dfrac{1+\lambda\mu}{\lambda}x^*-\mu y-\dfrac{1}
  {\lambda}z}_{-v}\right\|^2+\dfrac{\mu}{\lambda}\norm{z-y}^2\right]\\
  &=\dfrac{\lambda}{1+\lambda\mu}\left[\norm{v}^2+\dfrac{\mu}{\lambda}\norm{z-y}^2\right].
 \end{align}
 
 On the other hand, we also have $x^*-y=\frac{1}{1+\lambda\mu}(z-y)-\frac{\lambda}{1+\lambda\mu}v$,
 which in turn gives
 \begin{align}
 \label{eq:inner.v}
 \inner{v}{x^*-y}=\dfrac{1}{1+\lambda\mu}\left[\inner{v}{z-y}-\lambda\norm{v}^2\right].
 \end{align}
 Direct use of \eqref{eq:opt.value}, \eqref{eq:square.xstar} and \eqref{eq:inner.v} yields
 
 \begin{align*}
  \min_x\,q(x)+\varepsilon&=
  \dfrac{1}{1+\lambda\mu}\left[\inner{v}{z-y}-\lambda\norm{v}^2\right]
  +\dfrac{\lambda}{2(1+\lambda\mu)}\left[\norm{v}^2+\dfrac{\mu}{\lambda}\norm{z-y}^2\right]\\
  &=\dfrac{1}{2\lambda(1+\lambda\mu)}
     \left[2\inner{\lambda v}{z-y}-\norm{\lambda v}^2+\lambda\mu
     \norm{z-y}^2\right]\\
  &= \dfrac{1}{2\lambda(1+\lambda\mu)}\left[(1+\lambda\mu)\norm{y-z}^2
  -\norm{\lambda v+y-z}^2\right] \\
  &= \dfrac{1}{2\lambda}\left[\norm{y-z}^2
  -\dfrac{\norm{\lambda v+y-z}^2}{1+\lambda\mu}\right],
 \end{align*}
 which then yields
 \begin{align*}
 \min_x\,q(x)&=\dfrac{1}{2\lambda}\left[\norm{y-z}^2
  -\dfrac{\norm{\lambda v+y-z}^2}{1+\lambda\mu}\right]-\varepsilon\\
       &=\dfrac{1}{2\lambda}\left[\norm{y-z}^2
  -\left(\dfrac{\norm{\lambda v+y-z}^2}{1+\lambda\mu}+2\lambda\varepsilon\right)\right].
 \end{align*}
 
 (c) This follows from (b) and Taylor's theorem applied to $q(\cdot)$.
\end{proof}

\mgap
\mgap


%

\def\cprime{$'$} \def\cprime{$'$}

\end{document}